\documentclass[11pt]{article}

\usepackage[english]{babel}
\usepackage{authblk}

\title{Cluster parking functions II: \\ $q,t$-dihedral sieving via diagonal coinvariants}

\author{Matthieu Josuat-Vergès}

\affil{Université Paris Cité, CNRS, IRIF}

\usepackage{amsmath,amssymb,amsthm,mathtools}
\usepackage[linktocpage,unicode]{hyperref}

\usepackage{tikz}
\usepackage{physics}

\usepackage{a4wide}

\usepackage{caption}
\usepackage{enumitem}
\usepackage{url}

\hypersetup{
    pdftoolbar=true,        
    pdfmenubar=true,        
    pdffitwindow=false,     
    pdfstartview={FitH},    
    pdftitle={},
    pdfauthor={},
    pdfsubject={},
    pdfkeywords={},
    pdfnewwindow=true,      
    colorlinks=true,       
    linkcolor=red,          
    citecolor=blue,        
    filecolor=magenta,      
    urlcolor=cyan,           
    unicode=true          
}

\newtheorem{theo}{Theorem}[section]

\newtheorem{lemm}[theo]{Lemma}
\newtheorem{prop}[theo]{Proposition}
\newtheorem{conj}[theo]{Conjecture}

\theoremstyle{definition}
\newtheorem{defi}[theo]{Definition}
\newtheorem{rema}[theo]{Remark}

\DeclareMathOperator{\Aut}{Aut}
\DeclareMathOperator{\dih}{Dih}
\DeclareMathOperator{\rot}{Rot}

\newcommand{\CC}{\mathcal{C}}
\newcommand{\RR}{\mathcal{R}}

\renewcommand{\SS}{\mathcal{S}}

\newcommand{\II}{\mathcal{I}}

\newcommand{\FFF}{\mathbf{F}}
\newcommand{\RRR}{\mathbf{R}}
\newcommand{\SSS}{\mathbf{S}}
\newcommand{\III}{\mathbf{I}}

\newcommand{\bU}{\mathbb{U}}
\newcommand{\bV}{\mathbb{V}}
\newcommand{\bW}{\mathbb{W}}

\newcommand{\RRRR}{\mathtt{R}}
\newcommand{\SSSS}{\mathtt{S}}

\DeclareMathOperator{\Fix}{Fix}

\DeclareMathOperator{\CPF}{CPF}
\DeclareMathOperator{\PF}{PF}
\DeclareMathOperator{\NC}{NC}

\begin{document}

\maketitle

\begin{abstract}
    In a previous work, we defined the complex of cluster parking functions.  On one side, they encode the type-refined enumeration of faces of the cluster complex, and on the other side, they have a reduced homology which is isomorphic to (ungraded) diagonal coinvariants. The goal of this work is to take into account the underlying dihedral symmetry.  We thus have a product of a dihedral group and a symmetric group (there is a precise conjecture in the case of other finite Coxeter groups, but we focus on symmetric groups because of technicalities about diagonal coinvariants beyond this case).  Under the action of the product group, the reduced homology of cluster parking functions is conjecturally isomorphic to diagonal coinvariants up to tensoring by a sign character of the dihedral group.  This isomorphism can be reformulated as a dihedral sieving phenomenon.  The main technical contribution is the definition of the dihedral automorphism group of cluster parking functions, and we discuss various features of the reduced homology character and its conjectural connection with diagonal coinvariants.
\end{abstract}

\section{Introduction}

\subsection{Parabolic types and dihedral symmetry in the generalized cluster complex}

The {\it generalized cluster complex} $\Gamma^{(m)}$ associated to a finite real reflection group $W$ and an integer $m \geq 1$ is a simplicial complex with many combinatorial properties.  It was originally defined by Fomin and Reading~\cite{fominreading}, then reformulated in the context of representation theory by Thomas~\cite{thomas} and Zhu~\cite{zhu}. We refer to the thorough monograph by Stump, Thomas and Williams~\cite{stumpthomaswilliams}.  In the case of the symmetric group $\mathfrak{S}_n$, it can be described explicitly in terms of certain dissections of the regular $(mn+2)$-gon.

In a joint work with Douvropoulos~\cite{douvropoulosjosuatverges1,douvropoulosjosuatverges2}, we defined the {\it parabolic type} of each face $f \in \Gamma^{(m)}$.  It is a certain conjugacy class in $W$ (but it can also be seen as a conjugacy class of parabolic subgroups).  For example, in type $A_{n-1}$ and $m=1$, we associate an integer partition of $n$ to each dissection of the regular $(n+2)$-gon (these dissections form the combinatorial realization of the cluster complex in this case). Concretely, $\lambda = (\lambda_1,\lambda_2,\dots)$ is defined by the rule that each inner $(i+2)$-gon inside the dissection gives a part $i$ in $\lambda$.  For example, the dissection
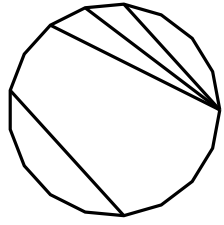
\[
    \begin{tikzpicture}[scale=1.4]
        \draw[very thick] (1.00, 0.000) -- (0.932, 0.361) -- (0.739, 0.674) -- (0.446, 0.895) -- (0.0923, 0.996) -- (-0.274, 0.962) -- (-0.603, 0.798) -- (-0.850, 0.526) -- (-0.983, 0.184) -- (-0.983, -0.184) -- (-0.850, -0.526) -- (-0.603, -0.798) -- (-0.274, -0.962) -- (0.0923, -0.996) -- (0.446, -0.895) -- (0.739, -0.674) -- (0.932, -0.361) -- (1.00,0.00);
        \draw[very thick,rounded corners=1] (1.00, 0.000) -- (0.0923, 0.996);
        \draw[very thick,rounded corners=1] (1,0) -- (-0.603, 0.798);
        \draw[very thick,rounded corners=1] (1.00, 0.000) -- (-0.274, 0.962);
        \draw[very thick,rounded corners=1] 
        (0.0923, -0.996) -- (-0.983, 0.184);
    \end{tikzpicture}
\]
of the heptadecagon gives the partition 15=6+4+3+1+1.  In general, we used a connection with parking functions to get an exact enumeration formula for the number of faces of each parabolic type, see~\cite[Corollary~7.3]{douvropoulosjosuatverges1}. 

The automorphism group of $\Gamma^{(m)}$ was described in~\cite{josuatverges}.  We will consider here the dihedral automorphism group generated by two elements $\RR$ and $\SS$ that satisfy $\RR^{mh+2}= (\RR\SS)^2 = \SS^2 = \II$ (see the definitions in Section~\ref{sec:dihedralcluster}). 

\begin{theo} \label{theo:invariant}
    For each (parabolic) conjugacy class $\CC$, the set of faces $f \in 
    \Gamma^{(m)}$ of parabolic type $\CC$ is stable under the action of the dihedral group $\langle \RR , \SS \rangle$.
\end{theo}


Rao and Suk~\cite{raosuk} introduced a notion of {\it dihedral sieving phenomenon}, encoding a dihedral action on a combinatorial set.  The previous theorem suggests investigating dihedral sieving in each set $\Gamma^{(m)}_{\CC}$, defined as the faces of $\Gamma^{(m)}$ of parabolic type $\CC$.  In the case $\CC = \{1_W\}$, $\Gamma^{(m)}_{\{1_W\}}$ contains all clusters (facets of $\Gamma^{(m)}$) and dihedral sieving phenomena have been investigated by Stier, Wellman, Xu~\cite{stierwellmanxu} through $q,t$-Catalan numbers.\footnote{Rao and Suk showed that dihedral sieving phenomena for the dihedral group $I_2(k)$ with odd $k$ can be encoded with symmetric polynomials in two variables, while the corresponding result for even $k$ is more complicated. Still, some actions of $I_2(k)$ with even $k$ can also be encoded by symmetric polynomial in two variables.  All actions considered in~\cite{stierwellmanxu} and in the present work are of this kind.}  {\it Cyclic} sieving phenomena for the action of $\RR$ on the sets $\Gamma^{(m)}_{\CC}$ have been obtained by Pouillart~\cite{pouillart1} in many cases.


\subsection{Extension of the dihedral symmetry to cluster parking functions}

The enumeration formula for the sets $\Gamma^{(m)}_\CC$ mentioned above can be interpreted using certain characters (in the parabolic Burnside ring of $W$) that are related to parking functions.  This led us to define {\it cluster parking functions}~\cite{douvropoulosjosuatverges2}. Concretely, an element of $\CPF^{(m)}$ is pair $(f,w W_{\underline{f}})$ where $f \in \Gamma^{(m)}$ and $w W_{\underline{f}}$ is a {\it parking function}, seen as a coset modulo a parabolic subgroup $W_{\underline{f}}$ that depends on $f$.  The action of $W$ is given by
\[
    u \cdot ( f, wW_{\underline{f}} ) 
    = 
    ( f, uwW_{\underline{f}} ).
\]
They form a simplicial complex, and by design their reduced homology character is a parking character twisted by the sign character.  More precisely, the reduced homology of $\CPF^{(m)}$ is nonzero only in degree $n-1$ (by the main result of~\cite{douvropoulosjosuatverges2} which says that $\CPF^{(m)}$ is a Cohen-Macaulay complex), and the  character of $\tilde H_{n-1}(\CPF^{(m)})$ as a $W$-space is
\begin{align}\label{eq:char_rw}
    w \mapsto \epsilon(w)\cdot (mh+1)^{\dim(\Fix(w))}
\end{align}
(where $\epsilon$ is the sign character, $h$ is the Coxeter number, and $\Fix(w)$ is the fixed point subspace of $w$ in the geometric representation, see Section~\ref{sec:parking} for details).  Up to the sign character, this is a character that appears in the rich theory of {\it parking spaces}~\cite{armstrongreinerrhoades,rhoades}.

Theorem~\ref{theo:invariant} suggests to extend the dihedral action on $\Gamma^{(m)}$ to cluster parking functions, in a way that is compatible with the action of $W$ and the simplicial structure.  Combinatorial models in types $A$, $B$, and $I$ also suggest that this is possible (see~\cite[Open problem~10.3]{douvropoulosjosuatverges2}).  The proof of Theorem~\ref{theo:invariant} actually provides an explicit answer to this question.

\begin{theo} \label{theo:extend}
     The automorphisms $\RRR,\SSS \in \Aut( \CPF^{(m)})$ defined in Section~\ref{sec:symCPF} are such that:
    \begin{itemize}
        \item $p \circ \RRR = \RR \circ p $ and $p \circ \SSS = \SS \circ p$, where $p$ is the projection $p(f,wW_{\underline{f}}) = f$,
        \item $\RRR$ and $\SSS$ commute with the action of $W$,
        \item $\RRR^{mh+2} = (\RRR\SSS)^2 = \SSS^2 = \III$ (i.e., they satisfy the same relations as $\RR$ and $\SS$, with the {\it caveat} that it is possible that $\RR^{(mh+2)/2} = \II$, but  $\RRR^{(mh+2)/2} \neq \III$ in general).
    \end{itemize}
    Assuming that the rank of $W$ is at least $2$, another pair $(\RRR',\SSS')$ satisfying the same conditions is such that $\RRR^{-1}\circ \RRR'$ (respectively $\SSS^{-1}\circ \SSS'$) is the action of a central element of $W$.
\end{theo}

Recall that the center of $W$ is either trivial (so that $\RRR,\SSS$ are unique) or is $\{e,w_0\}$ where $w_0$ is the long element.

A natural question is to investigate the action of $\langle \RRR,\SSS \rangle \times W$ on the reduced homology of $\CPF^{(m)}$.  Building on the idea of dihedral sieving via $q,t$-Catalan numbers, we are led to consider a dihedral action on diagonal coinvariants. 


\subsection{\texorpdfstring{Diagonal coinvariants and $q,t$-dihedral sieving phenomena}{Diagonal coinvariants and q,t-dihedral sieving phenomena}}
\label{sec:introdihedral}

Haiman's {\it diagonal coinvariants}~\cite{haiman} give a natural construction of a bigraded representation of the symmetric group $\mathfrak{S}_n$ of dimension $(n+1)^{n-1}$.  The associated Hilbert series in two variables $q$ and $t$ is subject to many combinatorial developments. 
There is an extension to other finite Coxeter groups by Gordon~\cite{gordon} (the naive generalization of diagonal coinvariants is too big, and one needs to consider a quotient defined via the rational Cherednik algebra).  For the rest of this introduction, we focus on the case of the symmetric group and assume $m=1$.  Let $\Gamma_n$ and $\CPF_n$  respectively denote the cluster complex and the complex of cluster parking functions in this case.

Let $A = \mathbb{C}[X_1,\dots,X_n, Y_1,\dots,Y_n]$ be the polynomial algebra in $2n$ variables.  The symmetric group acts diagonally by:
\begin{align*}
    \sigma \cdot X_i &= X_{\sigma^{-1}(i)}, \\
    \sigma \cdot Y_i &= Y_{\sigma^{-1}(i)}
\end{align*}
The dihedral group of order $2n+2$, denoted by $\dih_{n+2}$, acts on $A$ via two maps $\RRRR$ and $\SSSS$ defined by:
\begin{align} \label{eq:defRRRR}
    \RRRR( X_i )  &=  e^{2i\pi / (n+2)} X_i, \\
    \RRRR( Y_i )  &=  e^{- 2i\pi / (n+2)} Y_i
\end{align}
and
\[
    \SSSS(X_i) = Y_i, \qquad \SSSS(Y_i) = X_i.
\]
We observe that $A = \mathbb{C}[\rho_1\otimes\rho_2]$ where $\rho_1$, $\rho_2$ are the defining representation $\dih_{n+2}\to O(\mathbb{R}^2)$ and $\mathfrak{S}_n \to O(\mathbb{R}^n)$.

Let $I \subset A$ denote the ideal generated by constant-term-free $\mathfrak{S}_n$-invariant polynomials.  The quotient $A/I$ is the space of {\it diagonal coinvariants}, see Haiman~\cite{haiman}.  By homogeneity, it is straightforward to see that $I$ is invariant under the action of $\dih_{n+2}$, so that $\dih_{n+2}\times \mathfrak{S}_n$ acts on $A/I$.

\begin{conj}[$\Lambda\Sigma = \Delta$]  \label{conj1}
    We have $\Lambda\Sigma = \Delta$, where:
    \begin{itemize}    
        \item $\Lambda$ is the character of $\dih_{n+2}\times \mathfrak{S}_n$ acting on the reduced homology group $\tilde H_{n-2}(\CPF_n)$,
        \item $\Sigma$ is a sign character of $\dih_{n+2}$ extended to $\dih_{n+2}\times \mathfrak{S}_n$ (see Definition~\ref{def:sigma}),
        \item $\Delta$ is the character of $\dih_{n+2}\times \mathfrak{S}_n$ acting on diagonal coinvariants $A/I$.
    \end{itemize}
\end{conj}
    

It is interesting to reformulate the identity $\Lambda\Sigma = \Delta$ in terms of the bigraded character of $A/I$, that we denote by $\varDelta$.  This is a character of $\mathfrak{S}_n$ that polynomially depends on two variables $q$ and $t$ (moreover, it is symmetric in $q$ and $t$).  The equivalence between the two formulations is a straightforward consequence of the homogeneity of the ideal $I$.\footnote{This is not completely immediate in the case of reflections.  Indeed, $\SSSS$ isn't a mere rescaling of the variables.  But it becomes one when we consider new variables $X_i+Y_i$ and $X_i-Y_i$.} 

\begin{conj}[$q,t$-dihedral sieving] \label{conj2}
    For any $\FFF$ in $\dih_{n+2}$ and $\sigma\in\mathfrak{S}_n$, we have:
    \[
        (\Lambda\Sigma)(\FFF,\sigma) = \varDelta(\sigma)\Big\vert_{\substack{q=\lambda_1 \\ t=\lambda_2}}
    \]
    where $\lambda_1,\lambda_2$ are the eigenvalues of $\rho_1(\FFF) \in O(\mathbb{R}^2)$.
\end{conj}

This can be considered as a (homological) dihedral sieving phenomenon for cluster parking functions.  The various facets of this conjecture (including the general case beyond type $A$) will be discussed in Section~\ref{sec:dihedralsieving}.  Understanding this character identity leads to very different considerations, depending on wether $\FFF$ is a rotation or a reflection:
\begin{itemize}
    \item When we restrict the action to the rotation subgroup $\rot_{n+2} \subset \dih_{n+2}$, the evaluation of $\Delta$ gives explicit values (this follows from the fact that $\varDelta$ dramatically simplifies at $t = q^{-1}$).
    We get a variant of {\it parking spaces} considered by Armstrong, Reiner, and Rhoades~\cite{armstrongreinerrhoades}, with a cyclic group of order $n+2$ rather than $n$.   We call them {\it cluster parking spaces}.  It is not too difficult to generalize the conjecture to all finite Coxeter groups, and it has been proved in types $A$ and $B$ (see Pouillart~\cite{pouillart2}).
    \item The case of reflections of $\dih_{n+2}$ is connected with~\cite{corteeljosuatvergesvandenwyngaerd}, where we initiate the idea of plugging $t=-1$ in the context of $q,t$-combinatorics by building on the identity $\varDelta(e)\big\vert_{q=-t=1} = E_n$ ($e\in\mathfrak{S}_n$ is the unit and $E_n$ is the $n$th {\it Euler number}. In general, $(\Lambda\Sigma)(\SSS,1_W)$ is the {\it Springer number} of $W$, which is $E_n$ when $W = \mathfrak{S}_n$). This is the part of the conjecture that really relies on $q,t$-symmetric functions, but the lack of a well-developed $q,t$-combinatorics beyond the symmetric group makes it difficult to investigate this beyond symmetric groups.
\end{itemize}

As the main technical contribution of the present work is about automorphisms of $\CPF^{(m)}$ in the case of finite Coxeter groups, we hope that the natural analog of our conjectures (using Gordon's ring~\cite{gordon} in place of diagonal coinvariants, or the corresponding modules of the rational Cherednik algebra) also holds.   One should also investigate {\it higher diagonal coinvariants} defined by Garsia and Haiman~\cite[Section~5]{garsiahaiman} to deal with the general case $m\geq1$.



\subsection{Organization}

We review the generalized cluster complex in Section~\ref{sec:cluster}, its dihedral automorphism group in Section~\ref{sec:dihedralcluster}, and parabolic types of faces in Section~\ref{sec:parface} (in particular, the latter contains the proof of Theorem~\ref{theo:invariant}).  We review (noncrossing) parking functions in Section~\ref{sec:parking}, and cluster parking functions in Section~\ref{sec:symCPF}.  We also define the dihedral action on them in Section~\ref{sec:symCPF}, and the proof of Theorem~\ref{theo:extend} is completed in Section~\ref{sec:unicity}.  Conjectures~\ref{conj1} and~\ref{conj2} are discussed in Section~\ref{sec:dihedralsieving}, with a perspective in Section~\ref{sec:final}.

\subsection{Acknowledgements}

I deeply thank Theo Douvropoulos for sharing his insight throughout this project.  

\section{The generalized cluster complex}
\label{sec:cluster}

Let $W$ be a finite real reflection group of rank $n$ with $n\geq 2$, $V \simeq \mathbb{R}^n$ its geometric representation, and $\Phi$ the corresponding root system.  We choose a set $\Pi$ of simple roots and denote by $\Phi_+$ the corresponding set of positive roots.  Let $T\subset W$ denote the set of reflections.  For $\alpha\in \Phi$, the corresponding reflection is denoted $t_\alpha$.

\subsection{Steinberg's construction}

We fix a partition $\Pi = \Pi_\circ \uplus \Pi_\bullet$, where $\Pi_\circ$ (respectively, $\Pi_\bullet$) contains pairwise orthogonal roots. Let
\[
    c_\bullet = \prod_{\alpha \in \Pi_\bullet} t_{\alpha},\qquad
    c_\circ = \prod_{\alpha \in \Pi_\circ} t_{\alpha}.
\]
The {\it bipartite Coxeter element} is $c = c_\bullet c_\circ$.  The order of $c$ is the {\it Coxeter number}, denoted by $h$.

In a famous paper, Steinberg~\cite{steinberg} describes a recipe to compute all roots in terms of simple roots.  A consequence is that the whole root system can be decomposed as a disjoint union of $2h$ blocks:
\begin{align}
    \Phi = \biguplus_{i=1}^{2h} \Pi_{(i)},
\end{align}
where the blocks satisfy:
\begin{itemize}
    \item $\Phi_+ = \biguplus_{i=1}^{h} \Pi_{(i)}$, $\Pi_\bullet = \Pi_{(1)}$, $\Pi_\circ = \Pi_{(h)}$,
    \item $c_\bullet(\Pi_{(i)}) = \Pi_{(h+2-i)}$, $c_\circ(\Pi_{(i)}) = \Pi_{(h-i)}$, $-\Pi_{(i)} = \Pi_{(i+h)}$ (where indices are taken modulo $2h$).
\end{itemize}
Moreover, this decomposition has a nice interpretation when we project $\Phi$ onto the {\it Coxeter plane}: the roots in each $\Pi_ {(i)}$ have their projection in a common half-line issued from the origin.  (This projection on the Coxeter plane is often used to illustrate high-dimensional root systems.)


\subsection{Brady and Watt's construction}

The set of {\it almost-positive roots} is:
\[
    \Phi_{\geq -1} := \Phi_+ \uplus (-\Pi) = \biguplus_{i=0}^{h+1} \Pi_{(i)}.
\]
(Note that $\Pi_{(0)}=-\Pi_\circ$ and $\Pi_{(h+1)}=-\Pi_\bullet$.)  When $\Phi$ is crystallographic, this set naturally appears in the context of Fomin and Zelevinsky's finite type cluster algebras: $\Phi_{\geq -1}$ indexes {\it cluster variables} and it is the vertex set of the {\it cluster complex}~\cite{fominzelevinsky}.  Accordingly, this set also indexes the facets of a {\it generalized associahedron}.  Brady and Watt~\cite[Section~8]{bradywatt} used Steinberg's construction and minimal factorizations of the Coxeter element to give a purely Coxeter-theoretic derivation of the cluster complex. 

For $f \subset \Phi_{\geq -1}$ and $0\leq i \leq h+1$, let us denote
\[
    f_{(i)} := \prod_{\alpha \in f \cap \Pi_{(i)}} t_{\alpha}
\]
where $t_\alpha$ is the reflection associated with the root $\alpha$.  The factors are pairwise commuting, so that the product is well-defined and $f_{(i)}^2=1_W$.



Noncrossing partitions are well-known objects in Coxeter-Catalan combinatorics, and very useful to deal with the cluster complex.  Let us recall that the {\it reflection length} on $W$ is given by 
\[
    \ell(w) := 
    \min \big\{ 
        k\geq 0\;:\; \exists t_1,\dots,t_k \in T,\; w = t_1\cdots t_k
    \big\},
\]
and the {\it absolute order} on $W$ by
\[
    w_1 \leq w_2 
        \Longleftrightarrow
    \ell( w_1 ) + \ell( w_2 ) = \ell(w_1w_2),
\]
which says that $w_1$ lies in a minimal path from $1_W$ to $w_1w_2$ in the Cayley graph of $W$ (taking $T$ as generating set).  The interval $[1_W,c]$ for the absolute order turns out to be a ranked lattice, called the lattice on {\it noncrossing partitions}.  We denote it by $\NC$. 

Let $\NC_k$ denote the set of rank $k$ noncrossing partitions (the rank is the same as the absolute length).  The cluster complex, denoted by $\Gamma$, is an $(n-1)$-dimensional simplicial complex. Its faces are characterized by the following, which can be taken as a definition:

\begin{prop}[Brady and Watt~\cite{bradywatt}]
    Let $0\leq k \leq n$ and $f$ be a $k$-tuple of elements in $\Phi_{\geq -1}$.  Then $f\in \Gamma$ iff $ f_{(h+1)} \cdots f_{(1)} f_{(0)} \in \NC_k$.
\end{prop}

In particular, facets of $\Gamma$ are the $n$-tuples $f$ in $\Phi_{\geq -1}$ such that $ f_{(h+1)} \cdots f_{(1)} f_{(0)} = c$.  Note also that Brady and Watt considered a {\it total} order on $\Phi_{\geq -1}$.  This is equivalent to the present formulation upon choosing arbitrary total orders on $\Pi_\circ$ and on $\Pi_\bullet$, which turns out to be irrelevant since each of these set contains pairwise orthogonal elements (and the associated reflections commute).

\subsection{The Fuß extension}

We fix an integer $m\geq 1$ (it's called a {\it Fuß parameter}, because the analog of Catalan numbers counting facets of $\Gamma$ are called the {\it Fuß-Catalan numbers}).  A {\it colored root} is a pair $(\alpha,i)$ (denoted by $\alpha^i$ for brevity) with $\alpha\in \Phi$ and $1\leq i \leq m$.  A colored root $\alpha^i$ is {\it almost positive} if either $\alpha \in \Phi_+$, or $i=1$ and $-\alpha\in\Pi$. The set of {\it almost positive colored roots} is denoted by $\Phi^{(m)}_{\geq -1}$.  This is the vertex set of the {\it generalized cluster complex} $\Gamma^{(m)}$ introduced by Fomin and Reading~\cite{fominreading}.  

Tzanaki~\cite{tzanaki} extended Brady and Watt's result to this setting.  Let us describe the block decomposition of $\Phi^{(m)}_{\geq -1}$.  Denote $\Pi_\circ^i := \{ \alpha^i \;:\; \alpha\in\Pi_\circ \}$, and similarly for $\Pi_\bullet$, $-\Pi_\circ$, $-\Pi_\bullet$, $\Pi_{(j)}$ in place of $\Pi_\circ$.  We have 
\begin{equation}\label{eq:phi_decomp}
    \Phi^{(m)}_{\geq -1}
    =
    \biguplus_{i=0}^{mh+1}
    \varPi_{(i)},
\end{equation}
where 
\begin{itemize}
    \item $\varPi_{(0)} := (-\Pi_\circ)^1$,  $\varPi_{(mh+1)} := (-\Pi_\bullet)^1$,
    \item $\varPi_{(ih+j)} := \Pi_{(j)}^{m-i}$ if $0\leq i \leq m-1$ and $1\leq j \leq h$,
\end{itemize}
i.e., the whole sequence is given by
\begin{equation}
    \begin{aligned}
(\varPi_{(0)},\dots,\varPi_{(mh+1)})
    = 
    (-\Pi_\circ^1, \Pi_\bullet^m, \dots, \Pi_\circ^{m}, \Pi_\bullet^{m-1}, \dots, \Pi_\circ^{m-1}, \dots, \\ \Pi_\bullet^2, \dots, \Pi_\circ^2, 
    \Pi_\bullet^1, \dots, \Pi_\circ^1, -\Pi_\bullet^1)
    \end{aligned}
\end{equation}
It will always be understood that indices are taken modulo $mh+2$.

For each $f\subset \Phi^{(m)}_{\geq -1}$ and $0\leq i \leq mh-1$, we define
\[ 
    f_{(i)} := \prod_{\alpha \in f \cap \varPi_{(i)} } t_\alpha.
\]
(Of course, the reflection associated with a colored root is obtained by forgetting its color.)  The factors are again pairwise commuting, so that $f_{(i)}^2 = e$. 

Fomin and Reading defined $\Gamma^{(m)}$ as the flag complex associated to a binary relation.  Tzanaki's reformulation is what we take as a definition.

\begin{prop}[Tzanaki~\cite{tzanaki}] \label{prop:tzanaki}
Let $f$ be a $k$-tuple of elements in $\Phi^{(m)}_{\geq -1}$.  Then $f\in \Gamma^{(m)}$ iff $f_{(mh+1)} \cdots f_{(0)} \in \NC_k$.  In particular, an $n$-tuple $f$ is a facet of $\Gamma^{(m)}$ iff $f_{(mh+1)} \cdots f_{(0)} = c$.
\end{prop}

There is thus a noncrossing partition associated to each face of $\Gamma^{(m)}$.

\begin{defi}
    For each $f\in \Gamma^{(m)}$, we define
    \[
        \textstyle\prod f := f_{(mh+1)} \cdots f_{(0)} \in \NC.
    \]
\end{defi}

This notation is justified by the fact that $\prod f$ is the product of reflections $t_\alpha$ for $\alpha \in f$ (since each $t_\alpha$ appear in exactly one of the factors $f_{(i)}$.

\begin{rema}
    It is well-known that there is a bijection between clusters $f\in\Gamma=\Gamma^{(1)}$ and noncrossing partitions, both counted by Catalan numbers.  The map $\Gamma^{(m)} \to \NC$ that we just defined has a different nature: it's defined on all faces (not just clusters) and it's a projection rather than a bijection.  
\end{rema}

\section{Dihedral symmetry of the generalized cluster complex}
\label{sec:dihedralcluster}

We review the dihedral automorphism group of the generalized cluster complex, as described in~\cite{josuatverges}.  This group is generated by two elements $\RR, \SS$ such that $\RR^{mh+2} = \SS^2 = (\RR\SS)^2 = \II$ (the identity of $\Aut(\Gamma^{(m)})$).  Note also that the dihedral automorphism group might not be the full automorphism group $\Aut(\Gamma^{(m)})$ (because there might exist diagram automorphisms that are not in $\langle \RR ,\SS \rangle$, see~\cite{josuatverges} for details).

\begin{defi}[Fomin and Reading~\cite{fominreading}]
    The self-map $\RR :\Phi^{(m)}_{\geq -1} \to \Phi^{(m)}_{\geq -1} $ is defined by
    \begin{align} \label{eq:defR}
        \RR(\alpha^i) :=
        \begin{cases}
            \alpha^{i+1} & \text{ if } \alpha\in\Phi_+ \text{ and } i<m,\\
            (-\alpha)^1 & \text{ if } \alpha \in \Pi_\circ \text{ and } i=m, \text{ or } -\alpha \in \Pi_\bullet \text{ and } i=1,\\
            c(\alpha)^1 & \text{ if } \alpha \in \Phi_+ \backslash \Pi_\circ \text{ and } i=m, \text{ or } -\alpha \in \Pi_\circ \text{ and } i=1. 
        \end{cases}
    \end{align}
\end{defi}

The map $\RR$ was used in~\cite{fominreading} to define $\Gamma^{(m)}$.  By design, we have $\RR \in \Aut(\Gamma^{(m)})$.  Note that the map $\RR$ induces a simple action on the block decomposition~\eqref{eq:phi_decomp}:
\begin{align} \label{eq:rr_delta}
    \RR( \varPi_{(i)} ) = \varPi_{(i-h)},
\end{align}
where indices are taken modulo $mh+2$.

\begin{defi}[\cite{josuatverges}]
    The self-map $\SS$ on $\Phi^{(m)}_{\geq -1}$ is defined by:
    \begin{align} \label{eq:defS}
        \SS( \alpha^i )
        =
        \begin{cases}
            \alpha^i & \text{ if } \alpha \in -\Pi_\circ \text{ and } i=1 \text{ (i)}, \\
            (-\alpha)^m & \text{ if } \alpha \in -\Pi_\bullet \text{ and } i=1 \text{ (ii)}, \\
            (-\alpha)^1 & \text{ if } \alpha\in \Pi_\bullet \text{ and } i=m \text{ (ii')},\\
            \alpha^{m-i} & \text{ if } \alpha\in \Pi_\bullet \text{ and } 1\leq i \leq m-1 \text{ (iii)},\\
            c_\bullet(\alpha)^{m+1-i} & \text{ if } \alpha \in \Phi_+ \backslash \Pi_\bullet \text{ (iv)}.
        \end{cases}
    \end{align}
\end{defi}

It is also proved in~\cite{josuatverges} that $\SS\in\Aut(\Gamma^{(m)})$. Note that 
\begin{align} \label{eq:ss_delta}
    \SS( \varPi_{(i)} ) = \varPi_{(-i)}.
\end{align}

The composition $\SS\RR \in \Aut(\Gamma^{(m)})$ can be made explicit:
\begin{align} \label{eq:SSRR}
    \mathcal{SR}(\alpha^i) = 
    \begin{cases}
        \alpha^{m-i-1}  & \text{ if } \alpha \in \Pi_\bullet \text{ and } 1 \leq i \leq m-2 \;(i), \\
        c_\bullet(\alpha)^{m-i}  & \text{ if } \alpha \in \Phi_+ \backslash \Pi_\bullet \text{ and } 1 \leq i \leq m-1 \;(ii), \\
        (-\alpha)^1  & \text{ if } \alpha \in \Pi_\circ \text{ and } i = m \;(iii), \\
        (-\alpha)^m  & \text{ if } \alpha \in - \Pi_\circ \text{ and } i = 1 \;(iii'), \\
        (-\alpha)^{m-1}  & \text{ if } \alpha \in -\Pi_\bullet \text{ and } i = 1 \;(iv), \\
        (-\alpha)^1  & \text{ if } \alpha \in \Pi_\bullet \text{ and } i = m-1 \;(iv'), \\
        c_\circ (\alpha)^m & \text{ if } \alpha \in \Phi_+ \backslash \Pi_\circ \text{ and } i=m \;(v). 
    \end{cases}
\end{align}
By checking all cases, it is straightforward to check that $\SS\RR$ is an involution.  By~\eqref{eq:rr_delta} and~\eqref{eq:ss_delta}, we have:
\begin{equation} \label{eq:DeltaSSRR}
    \SS\RR( \varPi_{(i)}) = \varPi_{(h-i)}.
\end{equation}

\section{Parabolic types for the cluster complex}
\label{sec:parface}


By {\it parabolic type}, we mean an equivalence class of parabolic subgroup modulo conjugation.  For example, a parabolic subgroup of the symmetric group $\mathfrak{S}_n$ is isomorphic to a product $\mathfrak{S}_{\lambda_1}\times \mathfrak{S}_{\lambda_2} \times \dots$ where $\lambda$ is a partition of $n$.  This gives a correspondence between parabolic types and partitions of $n$.  One can also see a parabolic type as certain conjugacy classes in $W$:  a parabolic subgroup is associated to the conjugacy class of its Coxeter elements.

The parabolic type of $\pi\in\NC$ is the conjugacy class of $\pi$, or in terms of subgroups:
\[
    W_\pi := \{ w\in\ W \;:\; \Fix(\pi) \subset \Fix(w) \},
\]
where $\Fix(w) = \ker(1_W-w) \subset V$. (By Steinberg's theorem, $W_\pi$ is indeed parabolic.)  Enumeration of noncrossing partitions refined by parabolic type is an important problem related to {\it parking functions} (see~\cite{armstrongreinerrhoades} and next section).  In the case of $\mathfrak{S}_n$, the parabolic type of $\pi\in \NC$ corresponds to the partition of $n$ obtained by the cycle decomposition of $\pi$. (Or, seeing noncrossing partitions as certain set partitions, the parabolic type is obtained by taking the blocks sizes in decreasing order.)

Along the same lines, in~\cite{douvropoulosjosuatverges1,douvropoulosjosuatverges2} we associate a parabolic type to each face $f\in \Gamma^{(m)}$.  The definition involves the {\it Kreweras complement}, a lattice anti-automorphism $K:\NC\to\NC$ defined by
\[
    K(\pi) := \pi^{-1} c. 
\]

\begin{defi}[\cite{douvropoulosjosuatverges1}]
    For $f\in \Pi^{(m)}$, define 
    \begin{align}
        \underline{f} := K \big( \textstyle\prod f \big) = \big( \textstyle\prod f \big)^{-1} c \in \NC.
    \end{align}
    The {\it parabolic type} of $f$ is defined as the conjugacy class of the parabolic subgroup $W_{\underline{f}}$.
\end{defi}

(A slightly different definition of $K$ and $\underline{f}$ was used in~\cite{douvropoulosjosuatverges1}, but both lead to the same conjugacy class.  Also, recall that in the case of the symmetric group, this definition was illustrated in the introduction.) 

Note that $\underline{f}$ is a Coxeter element of the subgroup $W_{\underline{f}}$.  So, we can see the parabolic type of $f$ as the conjugacy class of $\underline{f}$, if we prefer to think of it in terms of a conjugacy class inside $W$.  In our previous work~\cite{douvropoulosjosuatverges1,douvropoulosjosuatverges2}, we showed that the parabolic type is an important invariant of faces $f \in \Gamma^{(m)}$.  A prominent result is an exact enumeration formula for faces refined by their parabolic type~\cite[Corollary~7.3]{douvropoulosjosuatverges1}.  It was also proved in~\cite[Proposition~11.1]{douvropoulosjosuatverges1} that the parabolic type is invariant under the action of $\RR$.  We review this here and complete the picture by proving Theorem~\ref{theo:invariant}.

\subsection{\texorpdfstring{The map $\RR$}{The map R}}
\label{sec:invaR}

We revisit the proof of \cite[Proposition~11.1]{douvropoulosjosuatverges1}. This is necessary to have an explicit conjugation equation between $\underline{f}$ and $\underline{f'}$, where $ f' := \RR(f) $. 
Using~\eqref{eq:rr_delta}, we have:
\[
    f'_{(i)} = \prod_{ \alpha \in \RR(f) \cap \varPi_{(i)}} t_\alpha = \prod_{ \alpha \in f \cap \varPi_{(i+h)}} t_{\RR(\alpha)}.
\]
From the definition of $\RR$ in~\eqref{eq:defR} (and $t_{w(\alpha)} = w t_\alpha w^{-1}$ for $\alpha\in\Phi_+$, $w\in W$), we have:
\[
    t_{\RR(\alpha)}
    =
    \begin{cases}
        t_\alpha & \text{ if } \alpha\in \varPi_{(i)} \text{ with } h \leq i \leq mh+1, \\
        c \cdot t_\alpha \cdot c^{-1} & \text{ if } \alpha\in \varPi_{(i)} \text{ with } 0 \leq i \leq h-1.
    \end{cases}
\]
From the previous two equations, we get
\begin{align}
    \label{eq:f'RR}
    f'_{(i)} = 
    \begin{cases}
        f_{(i+h)} & \text{ if } 0 \leq i \leq (m-1)h+1,\\
        c\cdot f_{(i+h)} \cdot c^{-1} & \text{ if } (m-1)h+2 \leq i \leq mh+1.
    \end{cases} 
\end{align}
So, 
\begin{align}
    \textstyle\prod f'  &= f'_{(mh+1)} \cdots f'_{(0)} \\
                    &=  c\cdot (f_{(h-1)}\cdots f_{(0)} ) \cdot c^{-1} \cdot f_{(mh+1)} \cdots f_{(h)}.
\end{align}
and it follows
\begin{align}
    \underline{f'}  
    &= 
    (f_{(h)} \cdots f_{(mh+1)}) \cdot c \cdot (f_{(0)} \cdots f_{(h-1)}) \\
    &=  
    (f_{(h-1)}\cdots f_{(0)}) \cdot \underline{f} \cdot (f_{(0)}\cdots f_{(h-1)}).
\end{align}
This shows in particular that $\underline{f'}$ and $\underline{f}$ are conjugate.

\subsection{\texorpdfstring{The map $\SS$}{The map S}}
\label{sec:invaS}

As for $\SS$, the computations are similar.  Let $f' := \SS(f)$.  From~\eqref{eq:ss_delta}, we have:
\[
    f'_{(i)} 
    =
    \prod_{ \alpha \in \SS(f) \cap \varPi_{(i)}} t_\alpha = \prod_{ \alpha \in f \cap \varPi_{(-i)}} t_{\SS(\alpha)}.
\]
From the definition of $\SS$ in~\eqref{eq:defS}, we have:
\begin{align} \label{eq:f'S}
    f'_{(i)} =
    \begin{cases}
        f_{(i)} & \text{ if } i=0, \\
        c_\bullet \cdot f_{(-i)} \cdot c_\bullet & \text{ otherwise. }
    \end{cases}
\end{align}
From the previous two equations, we get:
\begin{align}
    \textstyle\prod f' 
    &=
    c_\bullet \cdot (f_{(1)} \cdots f_{(mh+1)}) \cdot c_\bullet \cdot f_{(0)} \\
    &=
    c_\bullet \cdot f_{(0)} \cdot \big( \prod f \big)^{-1} \cdot c_\bullet \cdot f_{(0)}
\end{align}
so that
\begin{align}
    \underline{f'}
    &=
    f_{(0)} \cdot c_\bullet \cdot \big( \prod f \big) \cdot f_{(0)} \cdot c_\circ \\
    &=
    f_{(0)} \cdot c_\bullet \cdot  c \cdot \underline{f}^{-1} \cdot f_{(0)} \cdot c_\circ \\
    &= f_{(0)} \cdot c_\circ \cdot \underline{f}^{-1} \cdot f_{(0)} \cdot c_\circ.
\end{align}
Note that 
\[
    f_{(0)} \cdot c_\circ
    =
    \prod_{\alpha \in \Pi_\circ - f_{(0)}} t_\alpha,
\]
so $f_{(0)} \cdot c_\circ$ is an involution as a product of pairwise commuting reflections.  Moreover, each element of $W$ is conjugate to its inverse (see~\cite[Corollary~3.2.14]{geckpfeiffer}).  So, the previous equation shows that $\underline{f'}$ and $\underline{f}$ are conjugate.

\subsection{\texorpdfstring{The map $\SS\RR$}{The map SR}}

Although the previous two sections prove Theorem~\ref{theo:invariant}, it is useful to make the explicit computations regarding the composition $\SS\RR$.  The goal is again to have an explicit conjugation equation between $\underline{f}$ and $\underline{f'}$, where $f' := \SS\RR(f)$.  By~\eqref{eq:DeltaSSRR} we have:
\[
    f'_{(i)} 
    =
    \prod_{ \alpha \in \SS\RR(f) \cap \varPi_{(i)}} t_\alpha = \prod_{ \alpha \in f \cap \varPi_{(h-i)}} t_{\SS\RR(\alpha)}.
\]
Now, using the various cases in~\eqref{eq:SSRR} we get
\begin{align} \label{eq:SR}
    f'_{(i)} =
    \begin{cases}
        c_\bullet \cdot f_{(h-i)} \cdot c_\bullet & \text{ if } h+1 \leq h-i \leq mh+1, \\
        c_\circ \cdot f_{(h-i)} \cdot c_\circ & \text{ if } 0 \leq h-i \leq h.
    \end{cases}
\end{align}
And it follows
\begin{align}
    \prod f' 
    &= c_\bullet \cdot ( f_{(h+1)} \cdots f_{(mh+1)} ) \cdot c_\bullet \cdot c_\circ \cdot ( f_{(0)} \cdots f_{(h)}) \cdot c_\circ
\end{align}
and
\begin{align}
    \underline{f'} &=
    c_\circ \cdot ( f_{(h)} \cdots f_{(0)}) \cdot c_\circ \cdot c_\bullet \cdot ( f_{(mh+1)} \cdots f_{(h+1)} ) \cdot c_\circ \\
    &= c_\circ \cdot ( f_{(h)} \cdots f_{(0)}) \cdot c_\circ \cdot c_\bullet \cdot (\prod f) \cdot ( f_{(0)} \cdots f_{(h)} ) \cdot c_\circ \\
    &= c_\circ \cdot ( f_{(h)} \cdots f_{(0)}) \cdot \underline{f}^{-1} \cdot ( f_{(0)} \cdots f_{(h)} ) \cdot c_\circ.
\end{align}
As in the previous case, we deduce that $\underline{f}'$ and $\underline{f}$ are conjugate.

\begin{proof}[Proof of Theorem~\ref{theo:invariant}]
    We have seen in Sections~\ref{sec:invaR} and~\ref{sec:invaS} that the parabolic type of a face $f\in\Gamma^{(m)}$ is invariant under the action of $\RR$ and $\SS$.  So it is invariant under the group they generate.
\end{proof}

\section{Parking functions}
\label{sec:parking}

Let us review {\it noncrossing parking functions}, following Armstrong, Reiner and Rhoades~\cite{armstrongreinerrhoades}.  Noncrossing partitions have been briefly introduced in Section~\ref{sec:cluster}, and their parabolic types in Section~\ref{sec:parface}.  The problem of counting noncrossing partitions refined by their parabolic types naturally leads to {\it noncrossing parking functions}.  It should be emphasized that the theory in~\cite{armstrongreinerrhoades} features different kind of parking spaces, including a mysterious {\it algebraic parking space} which conjecturally unifies the theory.  In the present work, we mostly consider the one based on noncrossing partitions and allow ourselves to call parking functions what should be called a noncrossing parking function.

\begin{defi}
    A {\it parking function} is a coset $wW_\pi$ where $w\in W$ and $\pi\in \NC$.  Let $\PF$ denote the poset of parking functions, where the order is reverse inclusion.
\end{defi}

The Coxeter element $c$ generates a cyclic subgroup $\langle c \rangle \subset W$ of order $h$.  As the reflection length and absolute order on $W$ are invariant under the action of this cyclic group, there is an action on $\NC$ by conjugation:
\[
    c^i \cdot \pi = c^{i}\pi c^{-i}.
\]
This action carries over to an action on $\PF$ that commutes with the action of $W$.  Namely, the action of $\langle c \rangle \times W$ is:
\[
    (c^i, w_1) \cdot w_2W_{\pi} = w_1 w_2 W_{\pi}c^{-i}
\]
where $w_1, w_2 \in W$, $\pi\in \NC$.  This element is also a parking function because it can be written:
\[
    w_1 w_2 W_{\pi}c^{-i} = w_1 w_2 c^{-i} (c^{i} W_{\pi}c^{-i}) = w_1 w_2 c^i W_{ c^{i} \pi c^{-i}}, 
\]
and $\pi\in\NC$ implies $c^{-i} \pi c^i \in \NC$ as we just saw.

Let $\mu = e^{ 2i\pi/ h}$ be a primitive $h$th root of $1$.  It was conjectured in~\cite{armstrongreinerrhoades} that the character of $\PF$ is:
\[
    (c^k, w ) \mapsto \lim_{q \to \mu^k} \frac{\det(1_W - q^{h+1} w)}{\det(1_W -q w) }
    =
    (h+1)^{\dim\ker(1_W - \mu^k  w)}.
\]
Moreover, this has been checked in most cases ({\it ibidem}).  To our knowledge, it is still open only in types $E_7$ and $E_8$.




If we restrict the action to $W$, the orbit of $W_\pi$ (for $\pi\in\NC$) is the quotient $W/W_\pi$ seen as a $W$-set.  Computing the character of $\PF$ thus amounts to counting noncrossing partitions according to their parabolic type.  The relationship between this counting problem and the generalized cluster complex is discussed in~\cite{douvropoulosjosuatverges1}.

\section{Dihedral symmetry on CPF}
\label{sec:symCPF}

Cluster parking functions are the pairs $(f,wW_{\underline{f}}) \subset \Gamma^{(m)} \times \PF$, and the product order endows them with a partial order which can be interpreted as the inclusion of faces of a simplicial complex denoted $\CPF^{(m)}$.  The faces of a simplex $(f,wW_{\underline{f}}) \in \CPF^{(m)}$ are the elements $(f',wW_{\underline{f'}})$ where $f' \subset f$ (in particular, they are naturally in correspondence with faces of $f$).  We refer to $\cite{douvropoulosjosuatverges2}$ for details.  

Let us give a concrete combinatorial definition in type $A$. Let $\Gamma_n$ (respectively, $\CPF_n$) denote the complex $\Gamma^{(m)}$ (respectively, $\CPF^{(m)}$) in type $A_{n-1}$ with $m=1$.  A face $f\in \Gamma_n$ is a dissection of the (regular) $(n+2)$-gon, {\it i.e.}, a set of pairwise noncrossing diagonals.  For example, 
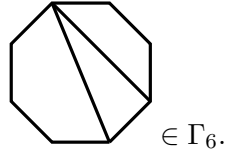
\[
    \begin{tikzpicture}
        \draw[very thick] (0.38,-0.92) -- (-0.38,-0.92) -- (-0.92,-0.38) -- (-0.92,0.38) -- (-0.38,0.92) -- (0.38,0.92) -- (0.92,0.38) -- (0.92,-0.38) -- (0.38,-0.92);
        \draw[very thick] (-0.38,0.92) -- (0.38,-0.92);
        \draw[very thick] (-0.38,0.92) -- (0.92,-0.38); 
    \end{tikzpicture} \in \Gamma_6.
\]
Moreover, $\langle \RR,\SS\rangle$ identifies to the symmetries of the polygon.  A face of $\CPF_n$ is obtained by putting labels $1,\dots,n$ in the inner polygons defined by a dissection, with the rule that a $(k+2)$-gon receives a set of $k$ (unordered) labels.  For example, 
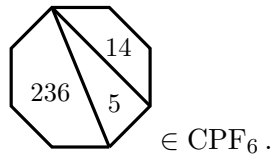
\[
    \begin{tikzpicture}
        \draw[very thick] (0.38,-0.92) -- (-0.38,-0.92) -- (-0.92,-0.38) -- (-0.92,0.38) -- (-0.38,0.92) -- (0.38,0.92) -- (0.92,0.38) -- (0.92,-0.38) -- (0.38,-0.92);
        \draw[very thick] (-0.38,0.92) -- (0.38,-0.92);
        \draw[very thick] (-0.38,0.92) -- (0.92,-0.38); 
        \node at (-0.4,-0.2) {\small 236};
        \node at (0.45,-0.35)  {\small 5};
        \node at (0.5,0.4)  {\small 14};
    \end{tikzpicture} \in \CPF_6.
\]
The symmetric group acts by natural relabeling.  Let us illustrate Theorem~\ref{theo:extend} before proving it: as a consequence of the uniqueness of $\RRR,\SSS$, the dihedral group $\langle \RRR,\SSS\rangle$ will again naturally identify to the group of symmetries of the regular $(n+2)$-gon (where diagonals and labels follow the movement of a symmetry of the polygon).  See also~\cite{douvropoulosjosuatverges2} for combinatorial models in types $B$ and $I$. 

Let's go back to the general case.  For each $f\in \Gamma^{(m)}$, let us define:
\begin{align} 
    \label{eq:defkappa}
    \kappa_f & := f_{(0)} \cdots f_{(h-1)},  \\
    \label{eq:deflambda}
    \lambda_f & := f_{(0)} \cdot c_\circ, \\
    \label{eq:defmu}
    \mu_f & := (f_{(0)} \cdots f_{(h)}) \cdot c_\circ.
\end{align}
Although they have no clear significance, understanding the automorphisms of $\CPF^{(m)}$ amounts to manipulating these quantities and checking various relations they satisfy. To begin, the computations in Section~\ref{sec:dihedralcluster} give:
\begin{align}
    \underline{\RR(f)} &= \kappa_f^{-1} \cdot \underline{f} \cdot \kappa_f,\\
    \underline{\SS(f)} &= \lambda_f^{-1} \cdot \underline{f}^{-1} \cdot \lambda_f,\\
    \underline{\SS\RR(f)} &= \mu_f^{-1} \cdot \underline{f}^{-1} \cdot \mu_f.
\end{align}
Since $W_{u v u^{-1}} = u W_v u^{-1}$, similar relations also hold for the corresponding parabolic subgroups:
\begin{align} \label{eq:rel_para1}
    W_{\underline{\RR(f)}} &= \kappa_f^{-1} \cdot W_{\underline{f}} \cdot \kappa_f,\\
    \label{eq:rel_para2}
    W_{\underline{\SS(f)}} &= \lambda_f^{-1} \cdot W_{\underline{f}} \cdot \lambda_f,\\
    \label{eq:rel_para3}
    W_{\underline{\SS\RR(f)}} &= \mu_f^{-1} \cdot W_{\underline{f}} \cdot \mu_f.
\end{align}
(Note that $W_{\underline{f}^{-1}} = W_{\underline{f}}$, since an element and its inverse have the same 1-eigenspace.)

\begin{defi}
    The two self-maps $\RRR$ and $\SSS$ on $\CPF^{(m)}$ are defined by:
\begin{align} \label{eq:defRRR}
    \RRR(f,wW_{\underline{f}}) 
    &= (\RR(f),wW_{\underline{f}}\cdot \kappa_f), \\
    \label{eq:defSSS}
    \SSS(f,wW_{\underline{f}}) 
    &= (\SS(f),wW_{\underline{f}}\cdot \lambda_f).
\end{align}
\end{defi}

Alternatively (via~\eqref{eq:rel_para1} and~\eqref{eq:rel_para2}), these can be written:
\begin{align} \label{eq:defRRR2}
    \RRR(f,wW_{\underline{f}}) 
    &= (\RR(f),w \kappa_f W_{\underline{\RR(f)}}), \\
    \label{eq:defSSS2}
    \SSS(f,wW_{\underline{f}}) 
    &= (\SS(f),w \lambda_f W_{\underline{\SS(f)}}).
\end{align}
The latter expressions makes clear that the right-hand sides are in $\CPF^{(m)}$, but the former are more clearly unambiguous.

\begin{lemm}
    For $f \in \Gamma^{(m)}$, we have $\mu_f = \kappa_f \cdot \lambda_{\RR(f)}$.
\end{lemm}

\begin{proof}
    Let $f' = \RR(f)$. From~\eqref{eq:f'RR}, we have $f'_{(0)} = f_{(h)}$. From~\eqref{eq:deflambda}, we get $\lambda_{f'} = f'_{(0)} \cdot c_\circ = f_{(h)} \cdot c_\circ$. From~\eqref{eq:defkappa}, we get $ \kappa_f \cdot \lambda_{f'} = (f_{(0)} \cdots f_{(h)} ) \cdot c_\circ$.  By~\eqref{eq:defmu}, we obtain the result.
\end{proof}

We have:
\begin{align*}
    \SSS(\RRR(f,wW_{\underline{f}}))
    &=
    \SSS( \RR(f) , wW_{\underline{f}}\cdot \kappa_f ) \\
    &=
    ( \SS\RR(f) , wW_{\underline{f}}\cdot \kappa_f \cdot \lambda_{\RR(f)}) \\
    &=
    ( \SS\RR(f) , wW_{\underline{f}}\cdot \mu_f ),
\end{align*}
which means that the composition $\SSS\RRR$ is the natural map that can be defined by analogy with~\eqref{eq:defRRR} and~\eqref{eq:defSSS}.

\begin{lemm}
    The maps $\SSS$ and $\SSS\RRR$ are involutions.
\end{lemm}

\begin{proof}
    Let's begin with $\SSS$.  Let $f\in\Gamma^{(m)}$ and $f' = \SS(f)$.  By~\eqref{eq:f'S}, we have $f'_{(0)} = f_{(0)}$, so that:
    \[
        \lambda_{\SS(f)} = \lambda_{f}.
    \]
    Then note that
    \[
        \lambda_{f} = f_{(0)} c_\circ = \prod_{\alpha \in \varPi_\circ - f} t_\alpha
    \]
    is involutive as a product of pairwise commuting reflections.  So, $\lambda_{f} \lambda_{f'} = \lambda_f^2 = 1_W$.  We then have:
    \[
        \SSS^2( f , w W_{\underline{f}} )
        =
        \SSS( f' , w W_{\underline{f}} \cdot \lambda_f )
        =
        ( \SS^2(f) , w W_{\underline{f}} \cdot \lambda_f \cdot \lambda_{f'} )
        =
        ( f , w W_{\underline{f}} ).
    \]
    So, $\SSS$ is an involution.

    Let's do the same for $\SSS\RRR$.  Let $f' = \SS\RR(f)$.  By~\eqref{eq:SR} and the definition of $\mu_f$, we have
    \[
        \mu_{f'} = (c_\circ f_{(h)}c_\circ) \cdots (c_\circ f_{(0)}c_\circ) \cdot c_\circ
        =
        c_\circ \cdot ( f_{(h)} \cdots f_{(0)} ) = \mu_f^{-1}.
    \]
    It follows:
    \[
        (\SSS\RRR)^2(f,w W_{\underline{f}})
        =
        \SSS\RRR( \SS\RR(f) , w W_{\underline{f}} \mu_f)
        =
        ( f , w W_{\underline{f}} \mu_f \mu_{f'})
        =
        ( f , w W_{\underline{f}}).
    \]
    So, $\SSS\RRR$ is also an involution.
\end{proof}

\begin{lemm}
    For any $f\in\Gamma^{(m)}$, we have:
    \begin{align}
        \kappa_{f} \cdot \kappa_{\RR(f)} \cdots \kappa_{\RR^{mh+1}(f)}
        = \underline{f}^h.
    \end{align}
\end{lemm}

\begin{proof}
    Let us define a sequence $(u_i)_{0\leq i \leq (mh+2)h-1}$ by the rule
    \[
        u_{(mh+2)i+j} = c^i \cdot f_{(j)} \cdot c^{-i}
    \]
    if $0 \leq i < h$ and $0 \leq j < mh+2$ are integers.  Note that we have
    \begin{align} \label{eq:shiftconj}
        u_{i+(mh+2)} = c \cdot u_{i} \cdot c^{-1},
    \end{align}
    for any $i$ (indices are modulo $(mh+2)h$).  We first compute
    \begin{align*}
        u_{(mh+2)i} \cdot u_{(mh+2)i+1} \cdots u_{(mh+2)(i+1)-1}
        &= (c^i \cdot f_{(0)} \cdot c^{-i}) \cdots (c^i \cdot f_{(mh+1)} \cdot c^{-i}) \\
        &= c^i \cdot (\prod f)^{-1} \cdot c^{-i},
    \end{align*}
    so that the whole product is
    \begin{align}
        u_0 u_1 \cdots u_{(mh+2)h-1} 
        &= 
        (\prod f)^{-1} \cdot 
        \Big( c \cdot (\prod f)^{-1} \cdot c^{-1} \Big) \cdots 
        \Big( c^{h-1} \cdot (\prod f)^{-1} \cdot c^{-h+1} \Big)\\
        \label{eq:wholeproduct1}
        &=
        \Big((\prod f)^{-1} \cdot c \Big)^h = \underline{f}^h,
    \end{align}
    where we simplified by $c^{-i}c^{i+1} = c$ and $c^{-h+1} = c$.  To complete the proof, we will compute this product by gathering factors in another way.
    
    First, we show that 
    \begin{align} \label{eq:Rishit}
        \RR^i(f)_{(j)} = u_{hi+j},
    \end{align}
    where $0\leq i,j \leq mh+1$.  This is done by induction on $i$.  The case $i=0$ is clear because $u_i = f_{(i)}$ if $0\leq i \leq mh+1$, by definition.  The induction from $i$ to $i+1$ comes from~\eqref{eq:f'RR}.  More precisely, if $0\leq j \leq (m-1)h+1$ then the first case of~\eqref{eq:f'RR} gives:
    \[
        \RR^i(f)_{(j)} = \RR^{i-1}(f)_{(j+h)} = u_{h(i-1)+j+h} = u_{hi+j}.
    \]
    And if $(m-1)h+2 \leq j \leq mh+1$, the second case of~\eqref{eq:f'RR} gives:
    \[
        \RR^i(f)_{(j)} = c \cdot \RR^{i-1}(f)_{(j+h-mh-2)} \cdot c^{-1} = c\cdot u_{hi+j-mh-2} \cdot c^{-1}.
    \]
    (Equation~\eqref{eq:f'RR} is written with indices modulo $mh+2$, here $j+h-mh-2$ is the representative of $j+h$ in $\{0,\dots mh+1\}$). By~\eqref{eq:shiftconj}, the previous equation proves~\eqref{eq:Rishit}.  
    
    From the definition of $\kappa_f$ in~\eqref{eq:defkappa}, we then obtain:
    \begin{align}
        \kappa_{\RR^i(f)} = u_{hi} \cdot u_{hi+1} \cdots u_{h(i+1)-1}.
    \end{align}
    So, by gathering the factors we have:
    \[
        u_0 \cdots u_{(mh+2)h-1} 
        =
        \kappa_{f} \cdot \kappa_{\RR(f)} \cdots \kappa_{\RR^{mh+1}(f)}.
    \]
    This completes the proof.
\end{proof}

\begin{prop}
    We have $\RRR^{mh+2} = \III$.
\end{prop}

\begin{proof}
    By iterating the definition of $\RRR$, we have
    \[
        \RRR^i( f , w W_{\underline{f}})
        =
        (\RR^i(f) , w W_{\underline{f}} \cdot \kappa_f \cdot \kappa_{\RR(f)} \cdots \kappa_{\RR^{i-1}(f)} ).
    \] 
    From $\RR^{mh+2} = \II$ together with the previous lemma, we get
    \[
        \RRR^{mh+2}( f , w W_{\underline{f}})
            =
        (f , w W_{\underline{f}} \cdot \underline{f}^h).
    \]
    Note that $\underline{f} \in W_{\underline{f}}$, so that $W_{\underline{f}} \cdot  \underline{f}^h = W_{\underline{f}}$.
\end{proof}

\begin{lemm} \label{lemm_kappakappa}
    Let $f\in\Gamma^{(m)}$, and $f'\subset f$.
    We have:
    \[
        \kappa_f (\kappa_{f'})^{-1}  
            \in
        W_{\underline{f'}}
    \]
\end{lemm}

\begin{proof}
    We can assume $\dim(f') = \dim(f)-1$, as the general case follows by transitivity. For example, if $f''\subset f' \subset f$ and the dimensions differ by 1 (the general case with more than 3 terms is similar), the property for $f''$ and $f$ is deduced by writing:
    \[
        \kappa_f (\kappa_{f''})^{-1}
            =
        \kappa_f (\kappa_{f'})^{-1}\cdot
        \kappa_{f'}(\kappa_{f''})^{-1}
            \in
        W_{\underline{f'}}\cdot
        W_{\underline{f''}}
            =
        W_{\underline{f''}},
    \]
    where the last equality holds because $f''\subset f'$ implies $W_{\underline{f'}} \subset W_{\underline{f''}}$.  

    The case $\kappa_f = \kappa_{f'}$ is clear, so we can assume otherwise.  This means that the unique vertex in $f-f'$ is in $\bigcup_{i=0}^{h-1} \varPi_{(i)} $, by definition of $\kappa_f = f_{(0)} \cdots f_{(h-1)}$.

    Since $f'\subset f$, we have $\prod f' \leq \prod f$ in the absolute order.  More precisely, since the dimensions differ by one there is a cover relation $\prod f' \lessdot \prod f$.  Let $t\in T$ be defined by $(\prod f')t = \prod f$.  From $\prod f' \lessdot (\prod f')t$, we deduce that $t \leq K(\prod f') $ in the absolute order, so that $t \in W_{\underline{f'}}$.  Eventually, write
    \[
        f'_{(mh+1)} \cdots f'_{(0)} \cdot t
            =
        f_{(mh+1)} \cdots f_{(0)}.
    \]
    We can simplify by $f'_{(mh+1)} \cdots f'_{(h)} = f_{(mh+1)} \cdots f_{(h)} $ (which holds because the unique vertex of $f-f'$ is in $\bigcup_{i=0}^{h-1} \varPi_{(i)} $) to get 
    \[
        (\kappa_{f'})^{-1} t = \kappa_f^{-1},
    \]
    so that $\kappa_f (\kappa_{f'})^{-1} = t \in W_{\underline{f'}}$.
\end{proof}

\begin{lemm} \label{lemm_lambdalambda}
    Let $f\in\Gamma^{(m)}$, and $f'\subset f$.
    We have:
    \[
        \lambda_f (\lambda_{f'})^{-1}  
            \in
        W_{\underline{f'}}.
    \]
\end{lemm}

\begin{proof}
    By the same argument as in the previous lemma, we can assume that $\dim(f') = \dim(f)-1$.  The case  $\lambda_{f'} = \lambda_{f}$ is clear, so we assume otherwise which means that the unique vertex in $f - f'$ is in $\varPi_{(0)}$ (by definition of $\lambda_f = f_{(0)}c_\circ$).

    As in the previous lemma, we use the relation $\prod f' \lessdot \prod f$ to obtain $t\in W_{\underline{f'}}$ such that $(\prod f')t = \prod f$.  Now, we write 
    \[
        f'_{(mh+1)} \cdots f'_{(0)} \cdot t
            =
        f_{(mh+1)} \cdots f_{(0)}
    \]
    and simplify by $f'_{(mh+1)} \cdots f'_{(1)} = f_{(mh+1)} \cdots f_{(1)} $ (which holds because the unique vertex of $f-f'$ is in $\varPi_{(0)} $) to get 
    \[
        f'_{(0)} t = f_{(0)}.
    \]
    Eventually, we have
    \[
        \lambda_{f} (\lambda_{f'})^{-1}
            =
        f_{(0)} c_\circ \cdot c_\circ (f'_{(0)})^{-1}
            =
        t \in W_{\underline{f'}},
    \]
    so that $\lambda_{f} (\lambda_{f'})^{-1} \in W_{\underline{f'}}$.
\end{proof}

\begin{prop}
    The maps $\RRR$ and $\SSS$ preserves the simplicial structure of $\CPF^{(m)}$, so that $\langle\RRR,\SSS\rangle \subset \Aut(\CPF^{(m)})$.
\end{prop}

\begin{proof}
    Let us prove the result for $\RRR$.  The result for $\SSS$ is obtained by writing $\lambda$ in place of $\kappa$, and using Lemma~\ref{lemm_lambdalambda} in place of Lemma~\ref{lemm_kappakappa}.

    Let $(f,wW_{\underline{f}})$ be a face of $\CPF^{(m)}$.  A smaller face can be written $(f',wW_{\underline{f'}})$ for $f'\subset f$.  We want to show:
    \[
        \big(\RR(f'), wW_{\underline{f'}}\kappa_{f'}\big) 
            \leq
        \big( \RR(f), wW_{\underline{f}}\kappa_f\big)
    \]
    for the product order on $\Gamma^{(m)}\times\PF$.  Since $\RR\in\Aut(\Gamma^{(m)})$, it remains to prove 
    \begin{align} \label{eq:kappacosets}
        wW_{\underline{f}}
        \kappa_f
            \subset
        wW_{\underline{f'}}
        \kappa_{f'}.
    \end{align}

    From Lemma~\ref{lemm_kappakappa}, we have $\kappa_f \in W_{\underline{f'}} \kappa_{f'}$.  Also, recall that $f'\subset f$ implies $W_{\underline{f}} \subset W_{\underline{f'}}$.  So we obtain $W_{\underline{f}} \kappa_f \subset W_{\underline{f'}} \kappa_{f'}$, and deduce~\eqref{eq:kappacosets}.
\end{proof}

The proof of Theorem~\ref{theo:extend} is almost complete: the first two points are clear from the definition, the third point follows from the lemmas in this section.  Uniqueness (up to the actions of central elements) will be proved in the next section.

\section{\texorpdfstring{Automorphisms of $\CPF$}{Automorphisms of CPF}}
\label{sec:unicity}

Let's first describe some edges in the adjacency graph of the facets of $\CPF^{(m)}$.  Let's fix a cluster $f\in\Gamma^{(m)}$.  For each $\alpha\in f$, let $W_{\underline{f-\{\alpha\}}} = \{e , t_\beta \}$ where $\beta$ implicitly depends on $f$ and $\alpha$.  (Since $f-\{\alpha\}$ has codimension $1$, the parabolic subgroup $W_{\underline{f-\{\alpha\}}}$ has rank $1$.)
The $1$-codimensional faces of $(f,\{w\})$ are: $(f-\{\alpha\},\{w,wt_\beta\})$, for $\alpha \in f$. In the facet adjacency graph, the neighbors of $(f,\{w\})$ through $(f-\{\alpha\},\{w,wt_\beta\})$ are:
\begin{itemize}
    \item one neighbor $(f,\{w t_\beta\})$,
    \item for each cluster $f' = (f-\{\alpha\}) \cup \{\alpha'\}$ (a neighbor of $f$ in $\Gamma^{(m)}$), two neighbors $(f',\{w\})$ and $(f',\{wt_\beta\})$.
\end{itemize}

(To check this, note that the link of each 1-codimensional face in $\CPF^{(m)}(W)$ is isomorphic to $\CPF^{(m)}(\mathfrak{S}_2)$, which consists in $2m$ isolated points. 
 See~\cite[Proposition~3.6]{douvropoulosjosuatverges2}.)

\begin{lemm} \label{lemm:betagen}
    With the notation above, the $n$ values of $\beta$ (corresponding to the $n$ values of $\alpha$ for a fixed $f$) can be ordered as a sequence $\beta_1,\dots,\beta_n$ so that $c = t_{\beta_1} \cdots t_{\beta_n}$.
\end{lemm}

\begin{proof}
    By Proposition~\ref{prop:tzanaki}, the vertices of $f$ can be ordered as $\alpha_1,\dots,\alpha_n$ so that $c = t_{\alpha_1} \cdots t_{\alpha_n}$.  We have:
    \begin{align*}
        \underline{f-\{\alpha_i\}}
        &=
        \big(\textstyle\prod (f-\{\alpha_i\}) \big)^{-1} c 
        =
        \big(t_{\alpha_1} \cdots t_{\alpha_{i-1}} t_{\alpha_{i+1}} \cdots t_{\alpha_n} \big)^{-1} t_{\alpha_1} \cdots t_{\alpha_n} \\
        &= t_{\alpha_n} \cdots t_{\alpha_{i}} \cdots 
        t_{\alpha_n}.
    \end{align*}
    It follows that $\underline{f-\{\alpha_i\}} = \{e, t_{\alpha_n} \cdots t_{\alpha_{i}} \cdots t_{\alpha_n}\}$ (this is a rank 1 parabolic subgroup containing $\underline{f-\{\alpha_i\}}$). So, let us define $t_{\beta_i} = t_{\alpha_n} \cdots t_{\alpha_{i}} \cdots t_{\alpha_n} $.  We get: $t_{\beta_1} \cdots t_{\beta_n} = t_{\alpha_n} \cdots t_{\alpha_1} = c^{-1}$.
\end{proof}

\begin{lemm}
    Let $\FFF \in \Aut(\CPF^{(m)})$ such that $p\circ\FFF = p$.  If the rank of $W$ is at least 2, there exists $u \in W$ such that $\FFF$ is the action of $u$ on $\CPF^{(m)}$.
\end{lemm}

\begin{proof}
    Let's fix a facet $f\in \Gamma^{(m)}$.  We will first prove: 
    \begin{align} \label{eq:FFF1}
        \forall w\in W, \; \FFF(f,\{w\}) = (f,\{uw\}),
    \end{align}
    for some $u\in W$ that we can define by $\FFF(f,\{e\}) = (f,\{u\})$. This is true for $w=e$ by definition, and we show that if this is true for $w$ it is for $wt_\beta$ too (for each value of $\beta$).  Let's fix $\alpha$ and $\beta$.  Note that there is only one facet of the form $(f,\{w'\})$ which is adjacent to $(f,\{w\})$ through $(f-\{\alpha\},wW_{\underline{f-\{\alpha\}}})$, namely $(f,\{wt_\beta\})$.  Since $\FFF$ is an automorphism, it follows that $\FFF(f,\{wt_\beta\})$ is the unique facet of the form $(f,\{w'\})$ which is adjacent to $\FFF(f,\{w\}) = (f,\{uw\})$ through $\FFF(f-\{\alpha\},wW_{\underline{f-\{\alpha\}}}) = (f-\{\alpha\},uwW_{\underline{f-\{\alpha\}}})$.  So $\FFF(f,\{wt_\beta\}) = (f,\{uwt_\beta\})$.  This shows that~\eqref{eq:FFF1} is also true for $wt_\beta$.  A consequence of Lemma~\ref{lemm:betagen} is that $t_{\beta_1},\dots,t_{\beta_n}$ generate $W$, and this completes the proof of~\eqref{eq:FFF1}.

    It remains to prove that the element $u$ defined via~\eqref{eq:FFF1} actually doesn't depend on $f$.  Let $f' = (f-\{\alpha\})\cup \alpha'$ be a neighbor facet of $f$ (and $\beta$ as before).  We clearly have $\FFF(f-\{\alpha\},\{w,wt_\beta\}) = (f-\{\alpha\},\{u w , u w t_\beta\})$.  Since $\FFF$ preserves incidence, we have either $\FFF(f',\{w\}) = (f',\{u w\})$ or $\FFF(f',\{w\}) = (f',\{u w t_\beta\})$.  But we already know that~\eqref{eq:FFF1} holds with $f'$ and some $u'$ in place of $f$ and $u$.  Since the rank of $W$ is at least 2, $W$ is noncommutative and it is straightforward to see that the only coherent possibility is $\FFF(f',\{w\}) = (f',\{u w\})$ for all $w$.  Now, the result follows from connectedness of the adjacency graph of facets of $\Gamma^{(m)}$ (which holds because $\Gamma^{(m)}$ is a shellable complex).
\end{proof}

\begin{proof}[Proof of Theorem~\ref{theo:extend}]
    It remains only to prove the last statement about other pairs $(\RRR',\SSS')$.  Suppose that $(\RRR',\SSS')$ satisfy the same conditions as $(\RRR,\SSS)$ in the theorem.  

    We have $p\circ \RRR'\circ \RRR^{-1} = \RR \circ p \circ \RRR^{-1} = \RR \circ \RR^{-1} \circ p = p$.  By the previous lemma (using the hypothesis that the rank is at least 2), $\RRR'\circ \RRR^{-1}$ acts on $\CPF^{(m)}$ as the action of some $u\in W$.  But $\RRR'\circ \RRR^{-1}$ commutes with the action of $W$ by hypothesis on $\RRR$ and $\RRR'$.  It follows that $u$ is central.  Similarly, $\SSS^{-1} \circ \SSS'$ is the action of a central element $u'$.  
\end{proof}

\section{\texorpdfstring{The $q,t$-dihedral sieving phenomenon}{The q,t-dihedral sieving phenomenon}}
\label{sec:dihedralsieving}

In this section, we mostly focus on the case where $m=1$ and $W$ is the symmetric group $\mathfrak{S}_n$.  (The exception is Section~\ref{sec:rotations} where we deal with any $W$.)

\subsection{Signs} We begin this section by discussing signs.  In particular, we define the character $\Sigma$ and thus complete the statement of the conjectures in the introduction.

A general principle that guides us in understanding signs is the following.  By~\cite{baclawskibjorner}, evaluating of the reduced homology character gives the reduced Euler characteristic of a fixed-point subposet:
\[
    \Lambda(\FFF, \sigma)
        =
    (-1)^{n-1} \tilde\chi(\CPF_n^{(\FFF, \sigma)})
\]
where $\CPF_n^{(\FFF,\sigma)}$ is the subposet of fixed points in $\CPF_n$ under the action of $(\FFF,\sigma)$ (seen as the face poset of a simplicial complex, so that its reduced Euler characteristic is well-defined).  Let us mention that these character evaluations are usually called {\it Lefschetz numbers}. (This is the origin of the notation $\Lambda$. Note that a common definition of Lefschetz numbers in the literature gives the evaluation of $(-1)^{n-1}\Lambda$ in our notation.)

In general, $\CPF_n^{(\FFF,\sigma)}$ is not a subcomplex: a face can be a fixed point while its vertices are permuted in a nontrivial way.  Still, this subposet is expected to be the face poset of a Cohen-Macaulay simplicial complex, so that the sign of its reduced Euler characteristic is given by the parity of its dimension.  We call them fixed-point subcomplexes (and conjecture the property of being the face poset of a Cohen-Macaulay simplicial complex).

\begin{defi} \label{def:sigma}
    The character $\Sigma$ of $\dih_{n+2}$ is defined on the generators by:
    \begin{align*}
        \Sigma (\SSS) &= (-1)^{|\Pi_\bullet|},\\
        \Sigma (\SSS\RRR) &= (-1)^{|\Pi_\circ|}.
    \end{align*}
    (We identify $\dih_{n+2}$ with $\langle \RRR,\SSS\rangle$.)  The extension of $\Sigma$ to $\dih_{n+2}\times\mathfrak{S}_n$ is also denoted by $\Sigma$.
\end{defi}

In the proof of Proposition~\ref{prop:springer}, we will explicitly describe the fixed-point subcomplex of $\SSS$ and it will be clear that the sign of $\Lambda(\SSS,1_W)$ is related with $\Sigma(\SSS)$.

\begin{rema}
    It's unclear what should be the analog of $\Sigma$ in the ``Fuß'' case.  
\end{rema}

\begin{rema}
    For $\sigma\in\mathfrak{S}_n$, let $\ell_e(\sigma)$ (respectively, $\ell_o(\sigma)$) denote the number of cycles of even length (respectively, odd length) in $\sigma$.  It is an {\it empirical observation} that 
    \[
        \varDelta(\sigma)\big\vert_{q=-t=1}
    \]
    has the same sign as $(-1)^{\ell_e(\sigma)+\frac{n-\ell_o(\sigma)}{2}}$.  In particular, it is positive when $\sigma$ is the identity (indeed, $\varDelta(e)\big\vert_{q=-t=1} = E_n$, the $n$th Euler number).  It's tempting to think about what should be the Coxeter theoretic analog (for example, $\ell_e(\sigma)$ is the multiplicity of $-1$ as an eigenvalue in the geometric representation) but we have no available data to test any conjecture in this direction.
\end{rema}





\subsection{Rotations}
\label{sec:rotations}

We consider the case $m=1$, and a finite Coxeter group $W$.  Let $\CPF = \CPF^{(1)}$ denote the complex of cluster parking functions in this case.  Note that $\Sigma$ restricts to a sign character of the rotation subgroup $\rot_{h+2} \subset \dih_{h+2}$.

The bigraded character $\varDelta$ of Gordon's ring dramatically simplifies at $t=q^{-1}$ (see~\cite[Section~5]{gordon}): 
\begin{equation} \label{eq:spec_vardelta}
    \varDelta(w) \Big\vert_{t=q^{-1}}
        =
    q^{-nh/2} \cdot \epsilon(w) \cdot \frac{\chi_w(q^{h+1})}{\chi_w(q)}
\end{equation}
where $\chi_w(q) := \det(w-q1_W)$ is the characteristic polynomial of $w \in W$ in the geometric representation.  This is the relevant specialization, since a rotation in $\dih_{h+2}$ has eigenvalues $e^{2i\pi/d}$ and its inverse $e^{-2i\pi/d}$ ($d$ is a divisor of $h+2$ and the order of the rotation).

\begin{prop}
Let $d$ be a divisor of $h+2$ and $\mu = e^{2i\pi / d}$.  We have:
\begin{equation} \label{eq:spec_vardelta2}
    \varDelta(w) \Big\vert_{q= \mu =t^{-1}}
        =
    (-1)^n \cdot \mu^{-\frac{n(h+2)}{2}} \cdot (-h-1)^{\dim \ker( w - \mu 1_W) }.
\end{equation}
\end{prop}

\begin{proof}
    We have to evaluate the right-hand side of~\eqref{eq:spec_vardelta}.  We have:
    \[
        \frac{\chi_w(q^{h+1})}{\chi_w(q)}
            =
        \prod_\lambda \frac{q^{h+1}-\lambda}{q-\lambda}
    \]
    where the product is over the eigenvalues $\lambda$ of $w$.  We distinguish different cases.  First consider an eigenvalue $\lambda \neq \pm 1$, so that $\lambda^{-1}$ is a distinct eigenvalue. The two corresponding factors give
    \[
        \frac{q^{h+1}-\lambda}{q-\lambda} \cdot 
        \frac{q^{h+1}-\lambda^{-1}}{q-\lambda^{-1}}
    \]
    that we rewrite as
    \begin{equation} \label{eq:twofactors}
        \frac{q^{h+1}-\lambda^{-1}}{q-\lambda} \cdot 
        \frac{q^{h+1}-\lambda}{q-\lambda^{-1}}.
    \end{equation}
    \begin{enumerate}[label=(\roman*)]
        \item If $\lambda \notin \{\mu,\mu^{-1}\}$, plugging $q=\mu$ in~\eqref{eq:twofactors} (and using $\mu^{h+1} = \mu^{-1}$) gives  a factor
        \[
            \frac{\mu^{-1}-\lambda^{-1}}{\mu-\lambda} \cdot \frac{\mu^{-1}-\lambda}{\mu-\lambda^{-1}} 
                =
            (-\tfrac{1}{\lambda\mu}) \cdot (-\tfrac{\lambda}{\mu})
                =
            \mu^{-2}.
        \]
        \item If $\lambda \in \{ \mu, \mu^{-1} \}$, by taking the limit $q\to \mu$ in~\eqref{eq:twofactors} we get:
        \[
            (h+1) \mu^h \cdot \frac{\mu^{-1}-\mu}{\mu-\mu^{-1}} 
                = 
            (-h-1) \mu^{-2}.
        \]
    \end{enumerate}
    Now, consider the factors coming from eigenvalues $\lambda = 1$ and $\lambda = -1$. 
    \begin{enumerate}[label=(\roman*),start=3]
        \item If $\lambda=1$ and $\mu \neq 1$, $\frac{q^{h+1}-1}{q-1} \to \frac{\mu^{-1}-1}{\mu-1} = -\mu^{-1}$.
        \item If $\lambda= \mu = 1$, $\frac{q^{h+1}-1}{q-1} \to h+1$, that we write $-(-h-1)$.
        \item If $\lambda=-1$ and $\mu \neq -1$, $\frac{q^{h+1}+1}{q+1} \to \frac{\mu^{-1}+1}{\mu+1} = \mu^{-1}$.
        \item If $\lambda=\mu = -1$ (note that $h$ is then even), $\frac{q^{h+1}+1}{q+1} \to (h+1)\cdot (-1)^h = h+1$, that we write $(-h-1) \mu^{-1}$.
    \end{enumerate}
    To gather the factors, observe that we have:
    \begin{itemize}
        \item a factor $\mu^{-1}$ for each eigenvalue of $w$ (sometimes coming in pairs, in case (i) and (ii)), so $\mu^{-n}$ in total,
        \item a factor $(-h-1)$ for each eigenvalue $\mu$ of $w$,
        \item there remains only to take into account a factor $-1$ in cases (iii) and (iv), i.e., when $\lambda=1$.  In total, this gives $(-1)^{\dim\Fix(w)}$. 
    \end{itemize}
    So, the right-hand side of~\eqref{eq:spec_vardelta} at $q=\mu$ is:
    \[
        \mu^{-nh/2-n} \cdot \epsilon(w) \cdot (-1)^{\dim\Fix(w)} \cdot (-h-1)^{\dim \ker( w - \mu 1_W) }.
    \]
    Then, recall that we have
    \[
        \epsilon(w) = (-1)^{n-\dim\Fix(w)}.
    \]
    So we get the announced formula.
\end{proof}

\begin{lemm}
    Let $d$ be a divisor of $h+2$.  Let $\mu$ be a root of $1$ of order $d$, and $\FFF \in \rot_{h+2}$ be a rotation of order $d$. We have:
    \begin{align}
        \mu^{-\frac{n(h+2)}{2}} = \Sigma(\FFF).
    \end{align}
\end{lemm}

\begin{proof}
    We only need to consider the case where $\mu = e^{\frac{2i\pi}{h+2}}$ and $\FFF = \RRR$ ({\it i.e.}, $d=h+2$), since other cases follow by taking powers thereof.  In this case, the left-hand-side is
    \[
        e^{-\frac{2i\pi}{h+2} \cdot \frac{n(h+2)}{2}}
            =
        e^{-i\pi n }
            =
        (-1)^n,
    \]
   and the right-hand side is
    \[
        \Sigma(\RRR) = \Sigma(\SSS) \cdot \Sigma(\SSS\RRR)  =  (-1)^{|\Pi_\bullet|} \cdot (-1)^{|\Pi_\circ|} = (-1)^n,
    \]
    since $|\Pi_\bullet| + |\Pi_\circ| = n$.
\end{proof}

The previous proposition and previous lemma give a reformulation of our conjecture in the case of rotations.

\begin{conj}[$\Lambda\Sigma = \Delta$, the case of rotations] \label{conj:rotcox}
    Let $\FFF$ be a rotation (a power of $\RRR$), and $\mu$ a root of $1$ with the same order as $\FFF$.  We have:
    \begin{equation} \label{eq:charrot}
        \Lambda(\FFF,w)
            =
        (-1)^n \cdot (-h-1)^{\dim\ker(w-\mu 1_W)}.
    \end{equation}
\end{conj}

\begin{rema}
    The ``Fuß'' analog of the previous conjecture is straightforward: take $\CPF^{(m)}$ and $-mh-1$ in place of $\CPF$ and $-h-1$.
\end{rema}

Since we have a simple explicit formula in the right-hand side, it is easy to test this case of the conjecture and one can imagine a case-by-case proof based on the finite type classification.

A confirmation of this conjecture would establish the existence of {\it cluster parking spaces} as an analog of the noncrossing parking spaces reviewed in Section~\ref{sec:parking}.  An important difference is that the cyclic group has order $h+2$ (or $mh+2$ in the Fuß case).  In the perspective of~\cite{armstrongreinerrhoades}, some natural questions are:
\begin{itemize}
    \item Besides clusters, does there exist other combinatorial construction of parking spaces having the same character ?  (In the same way as noncrossing parking spaces are related with nonnesting parking spaces, for example.)
    \item Is there an algebraic parking space, defined via homogeneous system of parameters, having the same character as the cluster parking spaces ?  This could lead to analogs of the strong/intermediate conjectures from~\cite{armstrongreinerrhoades} for cluster parking spaces. (About this, see Appendix~\ref{sec:app}.)
\end{itemize}

While these are quite speculative, the computations that lead to Conjecture~\ref{conj:rotcox} might be a clue that it's worth investigating these directions.

\begin{rema}
    To be precise, Conjecture~\ref{conj:rotcox} is only an analog of Conjecture~\ref{conj2} ($q,t$-dihedral sieving, {\it i.e.}, cyclic sieving in the present case).  To have an exact analog of Conjecture~\ref{conj1} ($\Lambda\Sigma=\Delta$), one needs to check that the operator $\RRRR$ from~\eqref{eq:defRRRR} is well-defined in the context of the rational Cherednik algebra and it simple module of dimension $(h+1)^n$ (which gives the appropriate generalization of diagonal coinvariants), so that there is indeed a character $\Delta$ of $\rot_{h+2}\times W$ beyond type $A_n$.  We will not do this here.
\end{rema}

\subsection{Reflections}

We first focus on the case of the symmetric group, where the character $\varDelta$ can be described through $q,t$-combinatorics.  More precisely, its Frobenius characteristic is $\nabla e_n$ (the image the $n$th elementary symmetric function $e_n$ under the Bergeron-Garsia operator $\nabla$) and it has an explicit expansion in terms of monomial symmetric functions given by the shuffle theorem~\cite{carlssonmellit}.  Using this, the values $\varDelta(\sigma)\vert_{t=-q=1}$ can be obtained via a computer for small $n$.  This confirmed that $\Lambda\Sigma = \Delta$ holds for $\mathfrak{S}_n$ if $n\leq 11$.

Let us now take the perspective from~\cite{corteeljosuatvergesvandenwyngaerd}, using the point of view of symmetric functions.  Let $e_\lambda$, $h_\lambda$, $p_\lambda$ respectively denote the elementary, homogeneous, power-sum symmetric functions.  We showed in~\cite{corteeljosuatvergesvandenwyngaerd} that
\[
    \langle \nabla e_n \mathrel{|} p_{1^n} \rangle \big\vert_{t=-1}
\]
is a nice $q$-analog of the $n$th Euler number $E_n$. (We also got other combinatorial results when replacing $\nabla$ with other operators.)  In particular, $\varDelta(1_{\mathfrak{S}_n})\big\vert_{t=-1,q=1} = E_n$.  This result relies on the schedule formula by Haglund and Sergel~\cite{haglundsergel}.  We don't know how to generalize this to the case of $\langle \nabla e_n | p_{\lambda} \rangle \big\vert_{t=-1,q=1}$, so let us just outline a few ideas. The polynomial $\langle \nabla e_n | p_{1^n} \rangle$ is the generating function of parking functions with respect to two well-known statistics: area and dinv (diagonal inversions).  The schedule formula gives a way to gather some terms (by taking parking functions with the same reading word) and get a power of $t$ times a product of $q$-integers, a quantity that simplifies nicely at $t=1$ and $q=-1$.  See~\cite{corteeljosuatvergesvandenwyngaerd,haglundsergel} for details.  A possible extension is the following.  The shuffle theorem gives $\langle \nabla e_n | h_{\lambda} \rangle$ as a generating functions of certain word parking functions with respect to area and dinv.  There is a so-called fermionic formula (see Haglund~\cite{haglund}) that generalizes the schedule formula: gathering word parking functions with the same reading word gives a power of $t$ times a product of $q$-binomial coefficients, so that the evaluation at $t=1$ and $q=-1$ is straightforward.  Eventually, this means that  
\[
    \langle \nabla e_n \mathrel{|} h_{\lambda} \rangle \big\vert_{t=1,\,q=-1}
\]
can be written as a (positive) combinatorial sum where each term is a product of binomial coefficients.  However, using this to get $\langle \nabla e_n | p_{\lambda} \rangle \big\vert_{t=1,q=-1}$ in an explicit way seems to be a formidable task.

We end this section by going back to the case of finite Coxeter groups, but in a very particular case.  

Let $S_\circ := \{t_\alpha \;:\; \alpha \in \Pi_\circ \} \subset S$ and $S_\bullet := S - S_\circ$.  Springer~\cite{springer} showed that the descent classes with maximal cardinality are those associated to $S_\circ$ and $S_\bullet$.  This cardinality is now called the {\it Springer number} of $W$, and denoted by $S_W$.

\begin{prop} \label{prop:springer}
    We have 
    \begin{align} \label{eq:sspringer0}
        \Lambda(\SSS , 1_W )
            &=
        (-1)^{|\Pi_\bullet| } S_W, \\
        \Lambda(\SSS\RRR , 1_W )
            &=
        (-1)^{|\Pi_\circ| } S_W.        
    \end{align}
\end{prop}

\begin{proof}
    Let's prove the first identity. The other follows because in the case $m=1$, the maps $\SSS$ and $\SSS\RRR$ are exchanged upon exchanging $\Pi_\circ$ and $\Pi_\bullet$.

    Let us first prove that the fixed points under $\SSS$ are the faces $(f,w W_{\underline{f}})$ where $f \subset \Pi_0 = -\Pi_\circ$.  We only need to check that the fixed point set of $\SS$ is $-\Pi_\circ$.  From Definition~\ref{eq:defS}, the only nontrivial case is that $c_\bullet(\alpha) \neq \alpha$ if $\alpha\in \Phi_+ \backslash \Pi_\bullet$.  A root $\alpha$ with $c_\bullet(\alpha)=\alpha$ would be orthogonal to the simple roots in $\Pi_\bullet$, and we let the reader check that this is not possible.      
    
    We get:
    \[
        \Lambda(\SSS , 1_W )
            =
        (-1)^{n-1} \tilde\chi\big( \CPF^{(\SSS,1_W)} \big)
            =
        (-1)^{n-1}
        \sum_{f \subset -\Pi_\circ} (-1)^{\dim(f)} | W/W_{\underline{f}}|.  
    \]
    Let $I \subset S_\circ$ and $f_I:=\{\alpha\in -\Pi_\circ \;:\; t_\alpha\in I\} $.  By definition of the Kreweras complement, we have:
    \[ 
        \underline{f_I}
            =
        (\prod f_I)^{-1} c
            =
        (\prod_{s\in I} s) \cdot c_\bullet \cdot c_\circ,
    \]
    so
    \[
        c_\circ \underline{f_I} c_\circ
            =
        (\prod_{s \in S_\circ - I} s) \cdot c_\bullet
    \]
    which is a standard Coxeter element of the standard parabolic subgroup $W_{S\backslash I}$.  This means that $W_{\underline{f_I}}$ is conjugate to $W_{S\backslash I}$.  So, we can rewrite the previous equality as:
    \[
        \Lambda(\SSS , 1_W )
            =
        (-1)^{n-1} \sum_{I \subset S_\circ} (-1)^{|I|-1} | W/W_{S\backslash I} |
            =
        (-1)^{|\Pi_\bullet|}
        \sum_{I \subset S_\circ} (-1)^{|S_\circ \backslash I|} | W/W_{S\backslash I} |.
    \]
    The number $| W/W_{S\backslash I} |$ can be interpreted as the number of elements of $W$ with descent set included in $I$.  The latter sum is thus an inclusion-exclusion to compute to cardinality of the descent class of $S_\circ$. This completes the proof of~\eqref{eq:sspringer0}.
\end{proof}


Using Definition~\ref{def:sigma} for the values of $\Sigma$, the previous proposition also gives
\[
    (\Lambda\Sigma)(\FFF,1_W) = S_W
\]
for any reflection $\FFF\in \langle \RRR,\SSS\rangle$.  This leads us to the following:

\begin{conj}[A very particular case of the $\Lambda\Sigma = \Delta$ conjecture]
    The bigraded character $\varDelta$ of Gordon's ring $R_W$ is such that
    \[
        \varDelta(1_W) \big\vert_{t=-q=1}
            =
        S_W.
    \]
\end{conj}

Studying the representation theory of rational Cherednik algebras could provide a proof of the previous statement.  This particular case is a good test to see if there's hope for $\Lambda\Sigma = \Delta$ in whole generality.

The previous conjecture has been confirmed in dihedral types by direct examination (in fact, we can use explicit computations of Boij and Geramita~\cite{boijgeramita} to check the full conjecture $\Lambda\Sigma=\Delta$).  It has been checked for $B_3$ using the {\it operator conjecture} from~\cite{haiman}.  More precisely, we obtain
\begin{align*}
    \varDelta(1_{B_3}) 
        = \,
    &q^9 + q^8t + q^7t^2 + q^6t^3 + q^5t^4 + q^4t^5 + q^3t^6 + q^2t^7 + qt^8 + t^9 + 3q^8 + 3q^7t + 3q^6t^2 +  \\
    & 3q^5t^3 + 3q^4t^4 +3q^3t^5 + 3q^2t^6 + 3qt^7 + 3t^8 + 5q^7 + 6q^6t + 6q^5t^2 + 6q^4t^3 + 6q^3t^4 + \\
    & 6q^2t^5 + 6qt^6 + 5t^7 + 7q^6 + 10q^5t + 10q^4t^2 + 10q^3t^3 + 10q^2t^4 + 10qt^5 + 7t^6 + 8q^5 + \\
    &  14q^4t + 14q^3t^2 + 14q^2t^3 + 14qt^4 + 8t^5 + 8q^4 + 15q^3t + 15q^2t^2 + 15qt^3 + 8t^4 + 7q^3 + \\ 
    & 12q^2t + 12qt^2 + 7t^3 + 5q^2 + 8qt + 5t^2 + 3q + 3t + 1,
\end{align*}
so that
\[
    \varDelta(1_{B_3})
    \big\vert_{t=-1,\, q=1}
        =
    11
\]
which is indeed the Springer number of type $B_3$.  As for $B_4$, the operator conjecture almost holds: the recipe gives a space of dimension $9^4-2$ instead of the expected $9^4$.  The bigraded Hilbert series of this space is a polynomial with 153 terms, that we don't reproduce here.  It turns out that the only missing term to get $\varDelta(1_{B_4})$ is $2t^4q^4$, which has been confirmed by computer calculations (using the definition in Appendix~\ref{app:rca}, and an AI-generated Sagemath program).  This allowed us to get
\[
    \varDelta(1_{B_4})
    \big\vert_{t=-1,\, q=1}
        =
    57
\]
which is indeed the Springer number of type $B_4$.  Similar computations give
\[
    \varDelta(1_{D_4})
    \big\vert_{t=-1,\, q=1}
        =
    23
\]
which is indeed the Springer number of type $D_4$.

\begin{rema}
    Leaving cluster parking functions aside, our computations suggest that in some cases,
    \[
        \varDelta(1_{W}) \big\vert_{t=-1}
    \]
    is a $q$-analog of the Springer number $S_W$ with positive coefficients and leading term $q^{nh/2}$.  This is known in type $A$ by~\cite{corteeljosuatvergesvandenwyngaerd}, and straightforward to check in dihedral types (for $I_2(k)$, we get $q^2\cdot\frac{1-q^{k-1}}{1-q}$).  The previous computations give:
    \begin{align*}
        \varDelta(1_{B_3}) \big\vert_{t=-1} &=
        q^9 + 2q^8 + 3q^7 + 3q^6 + 2q^5, \\
        \varDelta(1_{B_4}) \big\vert_{t=-1}
            &=
        q^{16} + 3q^{15} + 6q^{14} + 9q^{13} + 11q^{12} + 11q^{11} + 9q^{10} + 5q^{9} + 2q^{8}.  
    \end{align*}
    It would be highly interesting to confirm this positivity property and get a combinatorial (or geometric, or algebraic) interpretation of these $q$-Springer numbers of type $B$.  But the property doesn't hold in type $D$, where we get
    \begin{align*}
        \varDelta(1_{D_4}) \big\vert_{t=-1} &=
        q^{12} + 3q^{11} + 6q^{10} + 8q^{9} + 7q^{8} + 2q^{7} - 2q^{6} - 2q^{5}.
    \end{align*}\end{rema}

\section{Final remarks 
 and perspective}
\label{sec:final}

Our main technical contribution here is the definition of the dihedral automorphism group of cluster parking functions.  We hope that our conjectures will lead to interesting developments, and finish this article by a few remarks that we hope to be relevant in the sequel.
\begin{itemize}
    \item In the case of the symmetric groups (and $m=1$), it is likely that there exists an explicit generating function for the collection of numbers $\Lambda(\SSS,\sigma)$ (with $n\geq 0$, $\sigma\in\mathfrak{S}_n$).  It should be a symmetric functions given in terms of the homogeneous symmetric functions $h_\lambda$.  It is then reasonable to expect that symmetric functions theory will complete the proof of $\Lambda\Sigma = \Delta$.
    \item The first case to examine beyond symmetric groups is that of dihedral groups.  The $q,t$-character of diagonal coinvariants is known is a fully explicit way, see~\cite{boijgeramita}.  It is thus straightforward to check the identity $\Lambda\Sigma = \Delta$.
    \item We mentioned in the introduction the possibility of dihedral sieving phenomenon on each set $\Gamma^{(m)}_{\mathcal{C}}$.  We hope that these (in the case $m=1$, to begin with) will follow from $\Lambda\Sigma = \Delta$ by taking coefficients in an appropriate basis. 
    \item Let $\mathbb{Z}_k$ denote the cyclic group of order $k$.  The dihedral group $\dih_{h+2}$ is a semidirect product $\mathbb{Z}_{h+2} \rtimes \mathbb{Z}_2$. From the point of view of diagonal coinvariants, it is much more natural to consider the action of the direct product $\mathbb{Z}_{h+2} \times \mathbb{Z}_2$: instead of $\SSSS$ from the introduction, define 
    \[
        \SSSS'(X_i) = X_i, \qquad 
        \SSSS'(Y_i) = -Y_i,
    \]
    so that $\langle \RRRR,\SSSS' \rangle \simeq\mathbb{Z}_{h+2} \times \mathbb{Z}_2$.  Now, $\Delta$ becomes a character of $\mathbb{Z}_{h+2} \times \mathbb{Z}_2 \times \mathfrak{S}_n$.  But the relation $\Lambda \Sigma = \Delta$ still makes sense, because the characters of $\mathbb{Z}_{h+2} \rtimes \mathbb{Z}_2$ embed in those of $\mathbb{Z}_{h+2} \times \mathbb{Z}_2$ (upon identifying the underlying sets of the two groups).  This might be the correct path to state Conjecture~\ref{conj1} for finite Coxeter groups (where $\Delta$ should be defined via a module of the rational Cherednik algebra).
\end{itemize}


\appendix

\section{Toward an algebraic cluster parking space}
\label{sec:app}

Let $k$ and $n$ be positive integers.  We will build natural representations of $\mathbb{Z}_k \times \mathfrak{S}_n$ analogous to the right-hand side of~\eqref{eq:charrot}, using roots of unity.  It is interesting to put this in perspective of the theory of parking spaces~\cite{armstrongreinerrhoades}, which will be done at the end of this section.

\begin{rema}
    In this appendix, we consider the natural representation $\mathfrak{S}_n \hookrightarrow \operatorname{GL}(n,\mathbb{C})$ and the spectrum of a permutation is defined accordingly.  So, any quantity $\dim\ker(\sigma-\mu I)$ below should be thought in terms of this representation.  This differs from the results about parking functions above, where we considered the natural geometric representation (which has rank $n-1$ in the case of $\mathfrak{S}_n$).  This discrepancy will also be discussed at the end of this section.
\end{rema}

Let us define:
\begin{align*}
    \bU_k &:= \bigg\{ e^\frac{2i\pi u}{k} \; : \; 0\leq u <k  \bigg\}  \\[2mm]
    \bV_k &:= \bU_k \cup \{0\} \\[2mm]
    \bW_k &:= \bigg\{ f \in \mathbb{C}^{\bU_{k+2}} \;:\; \sum_{z \in \bU_{k+2}} f(z) = 0 \bigg\}.
\end{align*}
We identify the cyclic group $\mathbb{Z}_k$ with $\mathbb{U}_k$ in a natural way.  The set $\bV_k$ has cardinality $k+1$, and $\bW_k$ is a complex vector space of dimension $k+1$.  Note that $\bU_k$ acts on $\bV_k$ by multiplication, and $\bU_{k+2}$ acts on $\bW_k$ by
\[
    (\mu \cdot f)(z) = f(\mu z)
\]
for $(\mu,f) \in \bU_{k+2} \times \bW_k $. 

Now, $\bU_k \times \mathfrak{S}_n$ acts on the power set $\bV_k^n$ by:
\[
    (\mu,\sigma) \cdot ( z_1, \dots, z_n) 
        =
    ( \mu z_{\sigma^{-1}(1)}, \dots, \mu z_{\sigma^{-1}(n)}),
\]
{\it i.e.}, by the diagonal action of $\bU_k$ together with the permutation action of $\mathfrak{S}_n$.  

\begin{prop}[{\cite[Section~6.2]{armstrongreinerrhoades}}]
    The character of $\bU_k \times \mathfrak{S}_n$ acting on $\bV_k^n$ is 
    \[
        (\mu,\sigma) \mapsto (k+1)^{\dim\ker(\sigma - \mu I)}.
    \]
\end{prop}

Let us sketch a proof as a warm-up for the next proposition.

\begin{proof}
    This amounts to check that the number of fixed points under the action of $(\mu,\sigma)$ is 
    \[
        (k+1)^{\dim\ker(\sigma - \mu I)}.
    \]
    Let $(u_1,\dots,u_\ell)$ be a cycle of $\sigma$.  For $( z_1, \dots, z_n)$ to be fixed by $(\mu,\sigma)$, we need to have
    \[
        z_{u_i} 
            = 
        \mu z_{u_{i-1}} 
    \]
    for $1\leq i \leq \ell$ (and indices are modulo $\ell$).  This implies $z_{u_i} = \mu^{\ell} z_{u_i}$, by $\ell$ iterations.  There are two situations:
    \begin{itemize}
        \item If $\mu^{\ell} \neq 1$, the only possibility is $z_{u_1} = \dots = z_{u_\ell} = 0$.
        \item Otherwise ({\it i.e.}, the order of $\mu$ divides $\ell$), there are $k+1$ choices for $z_{u_1}$, and $z_{u_i} = \mu z_{u_{i-1}} $ forces the other values $z_{u_2}, z_{u_3}$, {\it etc}.
    \end{itemize}
    So, we get a power of $k+1$ and it remains to identify the exponent.  Let us represent the cyclic type of $\sigma$ by an integer partition $(\lambda_1,\dots,\lambda_m)$ of $n$. The spectrum of $\sigma$ is, as a multiset:
    \[
        \bigcup_{i=1}^m \bU_{\lambda_i}.
    \]
    It follows, denoting by $d$ the order of $\mu$:
    \[
        \#\big\{ i \;:\; d\mathrel{|}\lambda_i \big\}
            =
        \dim\ker(\sigma-\mu I).
    \]
    Eventually, the number of fixed point is $(k+1)^{\dim\ker(\sigma-\mu I)}$.
\end{proof}

Now, let us turn to the case of $\bW_k$.  

\begin{prop} \label{prop:algparkessai}
    The character of $\bU_{k+2} \times \mathfrak{S}_n$ acting on $\bW_k^{\otimes n}$ is
    \[
        (\mu, \sigma)
            \mapsto
        (-1)^n \cdot \epsilon(\sigma) \cdot (-k-1)^{\dim\ker(\sigma-\mu I)}.
    \]
\end{prop}

\begin{proof}
A basis of $\bW_k$ is given by the functions $f(z) = z^i$ for $1\leq i \leq k+1$ (for example, by the finite Fourier transform).  So, a basis of the tensor power $\bW_k^{\otimes n}$ is given by the elements
\[
    z^{\otimes I} := z^{i_1} \otimes \dots \otimes z^{i_n},
\]
where $I = (i_1,\dots,i_n) \in \{1,\dots,k+1\}^n$ is a multi-index. 

The action of $\mu\in\bU_{k+2}$ on the basis $(z^i)_{1 \leq i \leq k+1}$ is diagonal, since
\[
    \mu \cdot z^i = (\mu z)^i = \mu^i z^i.
\]
It follows that the action of $\bU_{k+2}$ on the tensor power $\bW_k^{\otimes n}$ is also diagonal, namely:
\[
    \mu \cdot z^{\otimes I} 
        =
    (\mu^{\sum I}) z^{\otimes I}
\]
where $\sum I = \sum_{j=1}^n i_j$.  On the other side, the natural action of $\mathfrak{S}_n$ permutes the basis elements:
\[
    \sigma \cdot z^{\otimes I}
        =
    z^{\otimes (\sigma\cdot I)}
\]
where $\sigma \cdot I = (i_{\sigma^{-1}(1)}, \dots, i_{\sigma^{-1}(n)} )$.  All in all, we see that the action of $\bU_{k+2} \times \mathfrak{S}_n$ on $\bW_k^{\otimes n}$ is monomial:
\[
    (\mu,\sigma) \cdot z^{\otimes I}
        =
    (\mu^{\sum I}) z^{\otimes (\sigma\cdot I)}.
\]
We thus get the trace of $(\mu,\sigma)$ as a sum:
\[
    \operatorname{Tr}(\mu,\sigma)
        =
    \sum_{\substack{I \in \{1,\dots,k+1\}^n, \\ \sigma\cdot I = I}} \mu^{\sum I}.
\]
Let us represent the cyclic type of $\sigma$ by an integer partition $(\lambda_1,\dots,\lambda_\ell)$ of $n$.  The sum can be rewritten:
\begin{align*}
    \operatorname{Tr}(\mu,\sigma)
        &=
    \sum_{(i_1,\dots,i_\ell) \in \{1,\dots,k+1\}^\ell} \mu^{\lambda_1 i_1 + \dots + \lambda_\ell i_\ell} \\[2mm]
        &=
    \prod_{u=1}^\ell \bigg(
        \sum_{i_u = 1}^{k+1} \mu^{\lambda_u i_u}.
    \bigg)
\end{align*}
To evaluate the sum between parentheses, note that:
\begin{itemize}
    \item if $\mu^{\lambda_u} = 1$ ({\it i.e.}, the order of $\mu$ divides $\lambda_u$), the sum is $k+1$,
    \item otherwise, $\mu^{\lambda_u}$ is a $(k+2)$nd root of unity different from $1$, so that 
    \[
        \sum_{i_u = 0}^{k+1} \mu^{\lambda_u i_u}
            = 
        0,
    \]
    and the sum from $1$ to $k+1$ is $-1$.
\end{itemize}
So we get:
\begin{align*}
    \operatorname{Tr}(\mu,\sigma)
        &=
    (-1)^{\ell - \dim\ker(\sigma-\mu I)} (k+1)^{\dim\ker(\sigma-\mu I)}.
\end{align*}
It remains to notice that $\epsilon(\sigma) = (-1)^{n-\ell}$.
\end{proof}

Now, let us briefly explain the idea of algebraic parking spaces from~\cite{armstrongreinerrhoades}.  Let us go back to the case of a finite real reflection group $W$ and its geometric representation $V$.  Let $x_1,\dots,x_n$ be a basis of $V$ and note that the action of $W$ naturally extends to polynomials in $x_1,\dots,x_n$.  A {\it homogeneous system of parameters} of degree $h+1$ is a sequence $\theta_1,\dots,\theta_n$ of homogeneous polynomials of degree $h+1$, such that the map $x_i \mapsto \theta_i$ commutes with the action of $W$ ($h$ is the Coxeter number of $W$).  Roughly speaking, the algebraic parking space associated to such $\theta_1,\dots,\theta_n$ is the set of solutions in $V\otimes \mathbb{C}$ of the system of equations $(x_i=\theta_i)_{1\leq i\leq n}$.  It is naturally endowed with an action of $\bU_h \times W$, and the character of this action is 
\[
    (\mu,w) \mapsto (h+1)^{\dim\ker(w-\mu I)}.
\]
See~\cite{armstrongreinerrhoades} for details.

In type $B_n$, the set $\bV_k^n$ for $k=2n+1$ is such an algebraic parking space, defined from $\theta_i = x_i^{2n+1}$.  In particular, what we considered above is the restriction to $\mathfrak{S}_n \subset \mathfrak{S}_n^\pm = B_n$ and this explains the fact that we considered the representation of $\mathfrak{S}_n$ in dimension $n$ rather than $n-1$.  We let the interested reader check that both propositions in this section can be extended from $\mathfrak{S}_n$ to $\mathfrak{S}_n^{\pm}$.  In fact, a peculiarity in this theory is the difficulty to find $\theta_1,\dots,\theta_n$ in type $A_n$, contrary to types $B$ and $D$ where a very simple and explicit one is given by $\theta_i = x_i^{h+1}$.  

At the end of Section~\ref{sec:rotations}, we asked if there exists an analog of the algebraic parking space that would lead to the character $(\mu,w) \mapsto (-1)^n \cdot (-h-1)^{\dim\ker(w-\mu I)}$.  But Proposition~\ref{prop:algparkessai} suggests to rather consider the character
\[
    (\mu,w) \mapsto (-1)^n \cdot \epsilon(w) \cdot (-h-1)^{\dim\ker(w-\mu I)}. 
\]
We hope that our result about $\bW_k^{\otimes n}$ is a clue toward such a construction, if it exists.

\section{\texorpdfstring{Explicit computations of $\varDelta$}{Explicit computations of ∆}}
\label{app:rca}

The goal of this section is to give a definition of the bigraded character $\varDelta$ which is as concrete as possible.  It has been used to compute some values of $\varDelta(1_W)$ and we hope to exploit this further in the future to get evaluations other than $1_W$.

We allow ourselves to follow notations from Gordon~\cite{gordon}, to which we refer for the general ideas behind this construction.  The general idea is that $\varDelta(1_W)$ is the bigraded Hilbert series of the simple module $L(\mathfrak{h}_0)$ of the rational Cherednik algebra $\mathcal{H}$ (up to tensoring by the sign character).

\newcommand{\h}{\mathfrak{h}}
\renewcommand{\H}{\mathcal{H}}
\newcommand{\C}{\mathbb{C}}

\subsection{Dunkl operators and rational Cherednik algebra}

Let $\h$ denote the reflection representation of $W$, and $\h_i = \bigwedge^i \h$ its $i$th exterior power.  In particular, $\h_0$ is the trivial representation and $\h_1 = \h$.  The {\it rational Cherednik algebra} $\H$ can be defined via a faithful representation in terms of {\it Dunkl operators}.  Let $\C[\h]$ denote the polynomial algebra on $\h$.  As an algebra of endomorphisms of $\C[\h]$, $\H$ is generated by:
\begin{itemize}
    \item elements of $\h$ acting by multiplication, 
    \item the Dunkl operators (whose definition is omitted here), indexed by $\h^*$,
    \item elements of $W$, by extending the action on $\h$ to polynomials on $\h$.
\end{itemize}
Let us mention that the definition of Dunkl operators depends on a parameter, which in the present context is $\frac{h+1}{h}$. 

\subsection{Presentation of the simple module}

A presentation of the simple module $L(\mathfrak{h}_0)$ is given by the exact sequence (which is a part of the Koszul resolution of $L(\mathfrak{h}_0)$):
\[
    M(\mathfrak{h}_1) \to 
    M(\mathfrak{h}_0) \to 
    L(\mathfrak{h}_0) \to 0
\]
where $M(\mathfrak{h}_1)$ and $M(\mathfrak{h}_0)$ are so-called {\it standard modules}.  We can identify $M(\mathfrak{h}_0)$ to the polynomial algebra $\mathbb{C}[\mathfrak{h}]$.  It remains to identify the ideal $I \subset \mathbb{C}[\mathfrak{h}]$ which is the image of $M(\mathfrak{h}_1)$.  From the theory of rational Cherednik algebras, we know that $I$ is generated by a subspace which is:
\begin{itemize}
    \item homogeneous of degree $h+1$,
    \item isomorphic the reflection representation $\mathfrak{h}_1$ under the action of $W$,
    \item in the kernel of all Dunkl operators ({\it i.e.}, it consists in {\it singular polynomials}).
\end{itemize}
(Note that the first two items can be formulated in terms of  homogeneous system of parameters from the previous section.)  

\subsection{Twisted periods and singular polynomials}

In order to make the presentation of $L(\h_0)$ more explicit, we need to describe the space of singular polynomials that generate the ideal $I$.  To do this we use the results of Feigin and Silantyev~\cite{feiginsilantyev}, who proved in particular that the space of singular polynomials that we need is unique.  Moreover, they give explicit construction of this space via {\it twisted periods}.  The idea is that there is a $W$-invariant polynomials $Q\in \mathbb{C}[\h]$ of degree $h+2$, called a twisted period, such that the space of first order derivatives 
\[
    \{ \partial_x Q \;: x\in \h \}
\]
is the generating space for the ideal $I$.  

Explicit constructions (in all possible degrees) have been given in~\cite[Section~5]{feiginsilantyev}.  In type $B_n$, the degree $h+2$ twisted period is given as a residue at $\infty$:
\[
    Q_{B_n} (x_1,\dots,x_n)
        =
    \operatorname{Res}_{z=\infty}
        \prod_{j=1}^n (z^2 - x_j^2)^{\nu} \dd{z}
\]
where $\nu = \frac{h+1}{h} = \frac{2n+1}{2n}$ and $x_1,\dots,x_n$ are the coordinates on which $B_n$ acts by signed permutations. Despite the fractional powers, the function is well-defined around $\infty$ and can be expanded as a series in $\frac{1}{z}$:  
\[
    \prod_{j=1}^n (z^2 - x_j^2)^{\nu} 
        =
    z^{2n\nu} \prod_{j=1}^n \Big(1 - \tfrac{x_j^2}{z^2}\Big)^{\nu} 
        =
    z^{2n+1} \prod_{j=1}^n \bigg(
        \sum_{k_j=0}^{\infty}
        \tbinom{\nu}{k_j}\big(-\tfrac{x_j^2}{z^2} \big)^{k_j} \bigg)
\]
The residue at $\infty$ is the opposite of the coefficient of $\tfrac{1}{z}$.  By expanding the product, this gives (up to a sign):
\[
    Q_{B_n}(x_1,\dots,x_n)
        =
    \sum_{\substack{ k_1,\dots,k_n \geq 0 \\ k_1 + \dots + k_n = n+1}} \prod_{j=1}^n \binom{\nu}{k_j} x_j^{2k_j}.
\]
Note that this polynomial contains a coefficient $\binom{\nu}{2n+2}x_i^{2n+2}$.  This means that $\partial_{x_i} Q_{B_n}$ contains a term in $x_i^{2n}$ and lower-degree terms in $x_i$.  It follows that a basis of $\mathbb{C}[\h]/I$ is given by monomials
\[
    \Big\{
        \prod_{i=1}^n x_i^{\ell_i} 
        \;:\;
        0\leq \ell_1,\dots,\ell_n \leq 2n
    \Big\}.
\]
Note that the number of such monomials is the expected dimension $(h+1)^n = (2n+1)^n$.

Similarly, it is stated in~\cite{feiginsilantyev} that
\[
    Q_{D_n} (x_1,\dots,x_n)
        =
    \operatorname{Res}_{z=\infty}
        z^{-2\nu} \prod_{j=1}^n (z^2 - x_j^2)^{\nu} \dd{z}
\]
where $\nu = \frac{h+1}{h} = \frac{2n-1}{2n-2}$.  This leads to 
\[
    Q_{D_n}(x_1,\dots,x_n)
        =
    \sum_{\substack{ k_1,\dots,k_n \geq 0 \\ k_1 + \dots + k_n = n}} \prod_{j=1}^n \binom{\nu}{k_j} x_j^{2k_j}.
\]
By the same argument, a basis of $\mathbb{C}[\h]/I$ is given by monomials
\[
    \Big\{
        \prod_{i=1}^n x_i^{\ell_i} 
        \;:\;
        0\leq \ell_1,\dots,\ell_n \leq 2n-2
    \Big\}
\]
and the number of such monomials is the expected dimension $(h+1)^n = (2n-1)^n$.

\subsection{Grading and filtration of the simple module}

There is a grading on $L(\h_0)$ coming from a grading on $\H$ (where $\h$ has degree $1$, the Dunkl operators have degree $-1$, and $W$ has degree $0$).  It is a shift of the usual grading of polynomials: the homogeneous component of degree $d$ in $L(\h_0)$, denoted $L(\h_0)_d$, is represented by the homogeneous polynomials $p(x_1,\dots,x_n)$ of degree $d + \frac{nh}{2}$ in $\mathbb{C}[\h]$.  The graded character of $L(\h_0)$ is (in a new variable $s$):
\[
    w \mapsto
    s^{-\frac{nh}{2}}\frac{\det(1_W - s^{h+1} w)}{ \det(1_W - s w)}.
\]

Additionally, there is a filtration on $L(\h_0)$ coming from a filtration on $\H$ (where $\h$ and the Dunkl operators have degree 1, and $W$ has degree $0$).  It is denoted
\[
    L(\h_0) = \bigcup_{i=0}^{nh/2} L(\h_0)^{(i)}. 
\]
The subspace $L(\h_0)^{(0)}$ is the 1-dimensional vector space generated by the {\it discriminant}, which is the polynomial
\[
    \delta(x_1,\dots,x_n) = \prod_{\alpha\in \Phi^+ } x_\alpha \in \mathbb{C}[\h]
\]
where $x_\alpha \in \mathbb{C}[\h]$ naturally corresponds to the positive root $\alpha$.  Then $L(\h_0)^{(i)}$ is defined inductively as the subspace generated by $L(\h_0)^{(i-1)}$ together with its images under the elements of $\h$ and the Dunkl operators.

The grading and the filtration are compatible in the sense that each subspace $L(\h_0)^{(i)}$ is homogeneous for the grading.  The natural Hilbert series in two variables $r$ and $s$ is then:
\[
    \varDelta(r,s) =
    \sum_{i=0}^{nh/2}
    r^i \cdot \operatorname{Hilb}\big(L(\h_0)^{(i)} / L(\h_0)^{(i-1)}, s\big)
\]
where the the Hilbert series in the variable $s$ is with respect to the grading.  To obtain a polynomial $\varDelta(q,t)$ in $q$ and $t$ which corresponds to the customary definition in type $A$ (and which is more natural from the point of view of diagonal coinvariants), we make the subtitution $r=\sqrt{qt}$, $s=\sqrt{q/t}$ to obtain $\varDelta(q,t)$.

In order to make this more computationally efficient, let us state the Poincaré-Birkhoff-Witt theorem: $\H$ is linearly isomorphic to 
\[
    \mathbb{C}[\h] \otimes
    \mathbb{C}[\h^*] \otimes
    \mathbb{C}[W],
\]
and in particular $\H^{(i)}$ (the $i$th term of the filtration) is linearly isomorphic to 
\[
    \bigoplus_{j+k\leq i}
    \mathbb{C}[\h]_j \otimes
    \mathbb{C}[\h^*]_k \otimes
    \mathbb{C}[W].
\]
Since $\delta$ is anti-invariant under the action of $W$, it follows that 
\begin{equation} \label{lh0i}
    L(\h_0)^{(i)}
        =
    \H^{(i)} \delta
        =
    \Big\langle
        p(x)q(y)\delta \;:\;
        p\in \mathbb{C}[\h],\; q\in \mathbb{C}[\h^*], \; \deg(p)+\deg(q)\leq i
    \Big\rangle /I.
\end{equation}

The next step is to restrict the polynomials $p$ and $q$ that are necessary. 

\begin{lemm}
    Suppose that $q(y)$ is a constant-term-free $W$-invariant polynomial.  Then $q(y)\delta = 0$.
\end{lemm}

\begin{proof}
    By invariance of $q(y)$,  $q(y)\delta$ is anti-invariant.  So it is in the ideal generated by $\delta$.  But $\deg( q(y)\delta ) < \deg(\delta)$ since $q(y)$ has no constant term.  So $q(y)\delta = 0$.    
\end{proof}

By a symmetry argument, we also have $p(x)q(y)\delta = 0$ when $p(x)$ is a constant-term-free $W$-invariant polynomial. So we can only consider polynomials $p(x)q(y)\delta$ where $p$ and $q$ runs through representatives of the coinvariant algebra.  

{\it In theory}, it is then possible to use the definition in \eqref{lh0i} to make explicit computations of $\varDelta(q,t)$.

\end{document}